\documentclass[a4paper]{amsart}
\usepackage{combelow}
\usepackage{enumerate,etex}
\usepackage{amssymb,hyperref,amsmath,amscd,color,mathrsfs,wasysym,stmaryrd}
\usepackage[all]{xy}
\RequirePackage{fix-cm}
\providecommand{\U}[1]{\protect\rule{.1in}{.1in}}

\newtheorem*{mainproblem}{Problem}
\newtheorem*{maintheorem}{Main Theorem}
\newtheorem*{pretheorem*}{Preliminary Version of Main Theorem}

\newtheorem*{proposition*}{Proposition}
\newtheorem*{definition*}{Definition}

\newtheorem*{theorem*}{Theorem}

\newtheorem{corollary}{Corollary}

\newtheorem{definition}{Definition}
\newtheorem{example}{Example}
\newtheorem{lemma}{Lemma}
\newtheorem{proposition}{Proposition}
\newtheorem{remark}{Remark}

\newcommand{\Lie}{\mathscr{L}}
\newcommand{\dd}{\mathrm{d}}
\newcommand{\bs}{\mathbf{s}}
\newcommand{\bt}{\mathbf{t}}
\newcommand{\im}{\operatorname{im}}
\newcommand{\R}{\mathbb{R}}
\newcommand{\diffto}{\xrightarrow{\raisebox{-0.2 em}[0pt][0pt]{\smash{\ensuremath{\sim}}}}}
\newcommand{\rmap}{\longrightarrow}
\newcommand{\lmap}{\longleftarrow}
\newcommand{\hmap}{\ensuremath{\lhook\joinrel\relbar\joinrel\rightarrow}}
\newcommand{\TT}{\mathbb{T}}
\newcommand{\la}{\langle}
\newcommand{\ra}{\rangle}
\newcommand{\pr}{\mathrm{pr}}

\begin{document}
\title{On dual pairs in Dirac geometry}
\author{Pedro Frejlich}
\address{Universidade Federal do Rio Grande do Sul, Campus Litoral Norte\\ Rodovia RS 030, 11.700 -- Km 92, Emboaba -- Tramanda\'i, RS, CEP 95590-000}
\email{frejlich.math@gmail.com}
\author{Ioan M\u{a}rcu\cb{t}}
\address{IMAPP, Radboud University Nijmegen, 6500 GL, Nijmegen, The Netherlands}
\email{i.marcut@math.ru.nl}
\begin{abstract}
In this note we discuss (weak) dual pairs in Dirac geometry. We show that this notion appears naturally when studying the problem of pushing forward a Dirac structure along a surjective submersion, and we prove a Dirac-theoretic version of Libermann's theorem from Poisson geometry. Our main result is an explicit construction of self-dual pairs for Dirac structures. This theorem not only recovers the global construction of symplectic realizations from \cite{CrMar11}, but allows for a more conceptual understanding of it, yielding a simpler and more natural proof. As an application of the main theorem, we present a different approach to the recent normal form theorem around Dirac transversals from \cite{BLM}.
\end{abstract}

\maketitle



\subsection*{Acknowledgments}

At a late stage of this project, we learned that E. Meinrenken was aware that the argument we present in the proof of our main theorem would work. We wish to thank him for his interest and kind encouragement, as well as many interesting discussions. We also wish the thank the referee for the detailed comments and suggestions, which greatly improved the quality of the paper.

The first author was supported by the Nederlandse Organisatie voor Wetenschappelijk Onderzoek (Vrije Competitie grant ''Flexibility and Rigidity of Geometric Structures'' 612.001.101) and by IMPA (CAPES-FORTAL project) and the second author by the Nederlandse Organisatie voor Wetenschappelijk Onderzoek (Veni grant 613.009.031) and the National Science Foundation (grant DMS 14-05671).

\tableofcontents
\newpage

\section{Introduction}

A \emph{symplectic realization}
\[
 \bs: (\Sigma,\omega^{-1}) \rmap (M,\pi)
\]
of a Poisson manifold $(M,\pi)$ is a surjective and submersive Poisson map $\bs$ from a symplectic manifold $(\Sigma,\omega)$\footnote{In the literature \cite{BursztynRadko,BCWZ,Weinstein83}, a \emph{symplectic realization} in our sense, in which the Poisson map is required to be submersive, is called a \emph{full symplectic realization}.}. The importance of symplectic realizations had been manifest since the early days of Poisson geometry (see \cite{Karasev,Weinstein83,CDW}).

Our main result concerns one of the possible Dirac-theoretic\footnote{Since the inception of Dirac geometry in the work of T.~Courant \cite{Courant}, many people worked on the field, and there are several good sources available. For background material on Dirac structures particularly close in spirit to the present note, we refer the reader to e.g.\ \cite{ABM,H,Marco,Eckhardnotes}. For the specific conventions and notations used in this paper, we refer the reader to Section \ref{sec : conventions}.} incarnations of the notion of symplectic realizations. Namely, we consider the following:
\begin{mainproblem}
Given a Dirac structure $L \subset TM\oplus T^*M$ on a manifold $M$, find a surjective submersion $\bs:\Sigma\to M$ and a closed two-form $\omega$ on $\Sigma$ such that
\[
 \bs: (\Sigma,\mathrm{Gr}(\omega)) \rmap (M,L),
\]
is a forward Dirac map (where $\mathrm{Gr}(\omega)$ denotes the graph of $\omega$).
\end{mainproblem}

In the Poisson case, a global, direct proof of the existence of symplectic realizations was presented in \cite{CrMar11}. Our result is a natural extension of this construction to the Dirac setting. Moreover, our construction produces a \emph{dual pair}, which is one of the main notions introduced in this paper, which generalizes to the Dirac setting the classical notion of dual pair in Poisson geometry:
\begin{definition*}
A {\bf dual pair} consists of surjective, forward Dirac submersions
\[(M_0,L_0)\stackrel{\bs}{\lmap} (\Sigma,\mathrm{Gr}(\omega)) \stackrel{\bt}{\rmap} (M_1,-L_{1}),\]
where $\omega$ is a closed two-form on $\Sigma$, such that the following hold:
\[\omega(V,W)=0,\quad
V\cap \ker\omega \cap W=0,\ \quad
\mathrm{dim}(\Sigma)=\mathrm{dim}(M_0)+\mathrm{dim}(M_1),\]
where $V:=\ker\bs_*$ and $W:=\ker \bt_*$.
\end{definition*}

The prototypical example of a dual pair is given by the presymplectic groupoid of an integrable Dirac structure \cite{BCWZ}.

We present the following solution to the Problem above (for a full statement of the result, including the explicit construction, see Section \ref{sec : On the existence of self-dual pairs}):

\begin{maintheorem}\label{thm : Dirac dual pairs}
Any Dirac manifold $(M,L)$ fits into a self-dual pair
\[(M,L) \stackrel{\bs}{\lmap} (\Sigma,\mathrm{Gr}(\omega)) \stackrel{\bt}{\rmap} (M,-L).\]
\end{maintheorem}

Here is an outline of the main results of the paper, and of its organization:
\vspace*{0.2cm}

\noindent \underline{Section \ref{sec : conventions}} establishes the notation and conventions used in the paper. Throughout, the symbol \S \ refers to this Section.
\vspace*{0.2cm}

\noindent \underline{Section \ref{sec : Pushing forward Dirac structures}} discusses in detail the problem of pushing forward a Dirac structure $L$ on $\Sigma$ to a Dirac structure on $M$ via a surjective submersion $\bs: \Sigma\to M$ with connected fibres. The prototypical statement is Libermann's theorem \cite{Libermann}, which states that a non-degenerate Poisson structure $\omega^{-1}\in \mathfrak{X}^2(\Sigma)$, corresponding to a symplectic structure $\omega \in \Omega^2(\Sigma)$, can be pushed forward through $\bs$ to a Poisson structure on $M$ if and only if the symplectic orthogonal to the vertical foliation $\ker \bs_*$ is involutive. We prove a general Dirac version of Libermann's theorem:
\begin{proposition*}[Dirac-Libermann]
A Dirac structure $L$ on $\Sigma$ can be pushed forward via $\bs$ to a Dirac structure on $M$ if and only if the Lagrangian family (see \S 5)\[p \mapsto L^{\bs}_p:=\bs^!(\bs_!(L_p))\subset T_p\Sigma\oplus T_p^*\Sigma\]is a Dirac structure on $\Sigma$ -- i.e.\ iff it forms a smooth, involutive subbundle of $\TT \Sigma$.
\end{proposition*}

Several criteria are discussed regarding smoothness and involutivity of $L^{\bs}$.
\vspace*{0.2cm}

\noindent \underline{Section \ref{sec : weak dual pairs}} first discusses the above proposition in the case when $L = \mathrm{Gr}(\omega)$ is the graph of a closed two-form $\omega \in \Omega^2(\Sigma)$. In particular, we show that a sufficient condition for a surjective submersion with connected fibres $\bs : \Sigma \to M$ to push forward a closed two-form $\omega \in \Omega^2(\Sigma)$ is the existence of an involutive subbundle $W \subset T\Sigma$ such that $L^{\bs}=V+\mathcal{R}_{\omega}(W)$, where $V:=\ker \bs_*$ and $\mathcal{R}_{\omega}(W)$ stands for the gauge-transformation of $W$ under $\omega$ (see \S 10). This leads naturally to the notions of \emph{weak dual pairs} and \emph{dual pairs} of Dirac structures, which we discuss in this section. Next, we give several equivalent descriptions of these notions, and present several illustrative examples.

\vspace*{0.2cm}

\noindent \underline{Section \ref{sec : properties}} discusses natural operations on weak dual pairs:
\begin{enumerate}[]
 \item\emph{composition} of two weak dual pairs with one common leg,
 \item\emph{pullback} of a weak dual pair through a surjective submersion,
 \item\emph{reduction} of a weak dual pair to a dual pair,
 \item\emph{pullback} of a (weak) dual pair via transverse maps,
\end{enumerate}
the latter to be used in the proof of the Main Theorem.

\vspace*{0.2cm}

\noindent \underline{Section \ref{sec : On the existence of self-dual pairs}} presents the complete statement and the proof  of the Main Theorem. This result generalizes the construction of symplectic realization for Poisson structures from \cite{CrMar11}; however, the proof presented here is completely conceptual, and bypasses all the unilluminating calculations in \emph{loc. cit.}, relying on natural Dirac geometric operators (like pullback, gauge transformations) which are not generally available in the Poisson setting.

\vspace*{0.2cm}

\noindent \underline{Section \ref{sec : dirac transversals}} provides, as an application of the Main Theorem, an alternative proof of the normal form theorem around \emph{Dirac transversals} from \cite{BLM}. In fact, the initial motivation for undertaking this project had been to prove a Dirac version of the normal form theorem of \cite{PT1} by extending the techniques from the Poisson setting; the outcome is the result stated below. In \cite{BLM}, this theorem is obtained through a completely different approach, by using a result on linearization of vector fields around submanifolds, which applies to other geometric settings (such as Lie algebroids, singular foliations, generalized complex structures etc.).

A \emph{Dirac transversal} in a Dirac manifold $(M,L)$ is an embedded submanifold $i:X \hookrightarrow M$ which is transverse to every presymplectic leaf of $(M,L)$. Such submanifolds inherit a Dirac structure $i^!(L)$.

\begin{theorem*}[Normal form around Dirac transversals]
Let $i:X \hookrightarrow M$ be a Dirac transversal in a Dirac manifold $(M,L)$, and let $p:NX \to X$ denote its normal bundle. Then, up to a diffeomorphism extending $i$ and an exact gauge-transformation, $L$ and $p^!i^!(L)$ are isomorphic around $X$.
\end{theorem*}
We refer the reader to Section \ref{sec : dirac transversals} for the detailed statement.

\vspace*{0.2cm}

\noindent \underline{Section \ref{sec : further}} indicates connections with other results from the literature, and presents possible implications and extensions of our work. The origin of the formula from the Main Theorem is explained in terms of the path-space approach to integrability of Dirac structures \cite{BCWZ,CrFer03}; we explain that our Main Theorem can be used to find explicit models for local presymplectic groupoids; we give the necessary ingredients to extend our results to twisted Dirac structures; and finally, we discuss the relation between dual pairs and Morita equivalence in Poisson geometry.

\section{Conventions and notation}\label{sec : conventions}

\begin{enumerate}[\S 1. ]
\item The standard Courant algebroid of a smooth manifold $M$ is denoted by $\TT M$, and $\mathrm{pr}_T,\mathrm{pr}_{T^*}$ denote the canonical projections:
\[
\TT M:=TM \oplus T^*M, \quad TM \stackrel{\mathrm{pr}_T}{\lmap} \TT M \stackrel{\mathrm{pr}_{T^*}}{\rmap} T^*M, \quad u \mapsfrom \ u+\xi \ \mapsto \ \xi.
\]
\item $\mathbb{T}M$ comes equipped with a nondegenerate, symmetric bilinear pairing
 \begin{align*}
\langle\cdot,\cdot\rangle:\TT M \times \TT M \rmap \mathbb{R}, \quad \langle u+\xi,v+\eta\rangle:=\iota_v\xi+\iota_u\eta,
 \end{align*}and the \emph{Dorfman bracket} at the level of sections:
 \begin{align*}\label{eq : dorfman bracket}
[\cdot,\cdot] : \Gamma(\mathbb{T} M) \times \Gamma(\mathbb{T} M) \rmap \Gamma(\mathbb{T} M), \quad   [u+\xi,v+\eta]:=[u,v]+\Lie_u\eta - \iota_v\dd\xi.
 \end{align*}
 \item For a smooth map $\varphi : M_0 \to M_1$, the elements $a_0=u_0+\xi_0 \in \TT_{x_0}M_0$ and $a_1=u_1+\xi_1 \in \TT_{x_1}M_1$ are \emph{$\varphi$-related} \cite[Definition 2.12]{Eckhardnotes}, denoted $a_0 \sim_{\varphi} a_1$, if
\[\varphi(x_0) = x_1, \quad \varphi_*(u_0)=u_1, \quad \xi_0 = \varphi^*\xi_1.\]
The notion of being $\varphi$-related applies to sections $a_0 \in \Gamma(\TT M_0)$ and $a_1 \in \Gamma(\TT M_1)$ as well, meaning that $a_{0,x_0} \sim_{\varphi} a_{1,\varphi(x_0)}$ for every $x_0 \in M_0$. Note that if $a_0,b_0 \in \Gamma(\TT M_0)$ and $a_1,b_1 \in \Gamma(\TT M_1)$ are such that
\[
 a_0 \sim_{\varphi} a_1, \quad b_0 \sim_{\varphi} b_1,
\]then also $[a_0,b_0] \sim_{\varphi} [a_1,b_1]$ (see e.g. the proof of \cite[Proposition 2.13]{Eckhardnotes}).

\item A subset $E \subset \TT M$ which meets each fibre $\TT_xM$ in a linear subspace $E_x:=E \cap \TT_xM$ is called a \emph{linear family}. Note that no continuity is assumed for the assignment $x \mapsto E_x$. The $\langle\cdot,\cdot\rangle$-orthogonal of a linear family $E$ is denoted by $E^{\perp} \subset \TT M$. If $E \subset E^{\perp}$, then $E$ is called an \emph{isotropic family}.
\item A linear family satisfying $E=E^{\perp}$ is called a \emph{Lagrangian family}; equivalently, if each $E_x \subset (\TT_xM,\langle\cdot,\cdot\rangle)$ is a maximal isotropic subspace. The set of maximal isotropic linear subspaces of $\TT_xM$ is denoted by $\mathrm{Lag}(\TT_x M)$.
\item A linear family $E \subset \TT M$ which is a smooth subbundle of $\TT M$ is  called \emph{smooth}; if in addition $E$ is Lagrangian, then it is called a \emph{Lagrangian subbundle}.

The graph of a two-form $\omega \in \Omega^2(M)$ and the graph of a bivector field $\pi \in \mathfrak{X}^2(M)$ are the following Lagrangian subbundles, respectively,
\begin{align*}
 \mathrm{Gr}(\omega) := \{ u+\iota_u\omega \ | \ u \in TM\}, \quad \mathrm{Gr}(\pi) := \{ \pi^{\sharp}\xi+\xi \ | \ \xi \in T^*M\}.
\end{align*}

\item An arbitrary linear family $E \subset \TT M$ is called \emph{involutive} if $[a,b] \in \Gamma(E)$ whenever $a,b \in \Gamma(E)$.
\item A \emph{Dirac structure} is an involutive Lagrangian subbundle. For instance, the Lagrangian subbundle $\mathrm{Gr}(\omega)$ defined by a two-form $\omega \in \Omega^2(M)$ is a Dirac structure iff $\omega$ is closed, whereas the Lagrangian subbundle $\mathrm{Gr}(\pi)$ defined by a bivector $\pi \in \mathfrak{X}^2(M)$ is a Dirac structure iff $\pi$ is Poisson.

\item The \emph{Courant tensor} of a Lagrangian subbundle $L$ is defined by
\[\Upsilon \in \Gamma(\wedge^3L^*), \quad \Upsilon(a_1,a_2,a_3):=\la [a_1,a_2],a_3\ra,\]and it vanishes exactly when $L$ is a Dirac structure \cite[Proposition 2.3.3]{Courant}.
\item The following \emph{operations on Dirac structures} will be used:
\begin{enumerate}[]
\item \emph{rescaling} a Dirac structure $L \subset \TT M$ by a scalar $\lambda \neq 0$:
\[
\lambda L := \{ u+\lambda \xi \ | \ u+\xi \in L\};
\]in particular, $-L=(-1)L$ denotes the \emph{opposite} Dirac structure;
\item \emph{gauge-transformation} by a closed two-form $\omega \in \Omega^2(M)$:
\[
\mathcal{R}_{\omega}(L) := \{ u+\xi+\iota_u\omega \ | \ u+\xi \in L\};
\]
\item \emph{pullbacks through submersions}: the pullback of a Dirac structure $L \subset \TT M$ through a submersion $\bs : \Sigma \to M$ is given by
\[
\bs^!(L) \subset \TT\Sigma, \quad \bs^!(L)_p = \{ u+\bs^*(\xi) \in \TT_p\Sigma \ | \ \bs_*(u)+\xi \in L_{\bs(p)}\}.
\]
\end{enumerate}
\item A smooth map $\varphi : M_0 \to M_1$ gives rise to pointwise operations of \emph{pullback} and \emph{push-forward}:
\begin{align*}
&\varphi^!: \mathrm{Lag}(\TT_{\varphi(x_0)} M_1) \rmap \mathrm{Lag}(\TT_{x_0} M_0), \quad \varphi^!(L_{1,x_1}):=\{a_0 \ | \ a_0 \sim_{\varphi} a_1 \in L_{1,\varphi(x_0)}\}\\
&\varphi_!: \mathrm{Lag}(\TT_{x_0} M_0) \rmap \mathrm{Lag}(\TT_{\varphi(x_0)} M_1), \quad \varphi_!(L_{0,x_0}):=\{a_1 \ | \ L_{0,x_0} \ni a_0 \sim_{\varphi} a_1\}.
\end{align*}
\noindent A smooth map between Dirac manifolds $\varphi : (M_0,L_0) \to (M_1,L_1)$ is called
\begin{enumerate}[]
\item \emph{forward} if $\varphi_!(L_{0,x_0}) = L_{1,\varphi(x_0)}$ for all $x_0 \in M_0$;
\item \emph{backward} if $L_{0,x_0}=\varphi^!(L_{1,\varphi(x_0)})$ for all $x_0 \in M_0$.
\end{enumerate}
\item In general, for a smooth map $\varphi : M_0 \to M_1$ and a Lagrangian subbundle $L_1$ on $M_1$, $\varphi^!(L_1)$ need not be smooth. However (see e.g.\ \cite[Proposition 5.6]{H}):
\begin{enumerate}[(a)]
 \item A sufficient criterion ensuring smoothness of $\varphi^!(L_1)$ is that $L_1 \cap \ker(\varphi^*)$ have constant rank;
 \item In particular, $\varphi^!(L_1)$ is smooth when $\varphi$ is \emph{transverse} to $L_1$, i.e., when either of the equivalent conditions holds:
 \begin{enumerate}[i)]
  \item $L_1 \cap \ker(\varphi^*) = 0$;
  \item for all $x_0 \in M_0$, we have that $\varphi_*(T_{x_0}M_0) + \mathrm{pr}_T(L_{1,\varphi(x_0)}) = T_{\varphi(x_0)}M_1$;
 \end{enumerate}
 \item When $L_1$ is a Dirac structure and $\varphi^!(L_1)$ is smooth, then it is automatically a Dirac structure.
\end{enumerate}
\item Let $A_1 \to M_1$ be a vector bundle, and let $\varphi : M_0 \to M_1$ be a smooth map. Denote by $\varphi^*(A_1):=M_0 \times_{M_1}A_1 \to M_0$ the pullback vector bundle. If $A_0 \to M_0$ is another vector bundle, a vector bundle map $\Phi : \varphi^*(A_1) \to A_0$ over $\mathrm{id}_{M_0}$ is called a \emph{comorphism of vector bundles}.

A comorphism of vector bundles $\Phi : \varphi^*(A_1) \to A_0$ induces a map on sections \[ \Phi^{\dagger} : \Gamma(A_1) \to \Gamma(A_0),\quad a_1\mapsto \Phi\circ a_1\circ \varphi.
\]

If $A_i \to M_i$ are Lie algebroids, with anchor maps $\varrho_i:A_i \to TM_i$, a comorphism of vector bundles $\Phi:\varphi^*(A_1) \to A_0$ is called a \emph{comorphism of Lie algebroids} if it is compatible with anchors, $\varphi_* \circ \varrho_0 \circ \Phi = \varrho_1$, and the induced map of sections $\Phi^{\dagger}$ is a Lie algebra homomorphism. A comorphism of Lie algebroids $\Phi$ is called \emph{complete} if the following condition is met: if $a_1 \in \Gamma(A_1)$ is such that $\varrho_1(a_1) \in \mathfrak{X}(M_1)$ is a complete vector field, then $\varrho_0(\Phi^{\dagger}a_1) \in \mathfrak{X}(M_0)$ is also a complete vector field. See e.g. \cite{Weins_comorph} for more details on comorphisms.
\end{enumerate}

\section{Pushing forward Dirac structures}\label{sec : Pushing forward Dirac structures}

In this section, we discuss the following problem: given a Dirac structure $L \subset \TT \Sigma$, when does a surjective submersion with connected fibres
\[\bs : \Sigma \rmap M\]push $L$ forward to a Dirac structure $L_M \subset \TT M$, i.e.\ when is there a Dirac structure $L_M$ for which $\bs : (\Sigma,L) \to (M,L_M)$ is a forward Dirac submersion?

First observe that, for a given $L_M \subset \TT M$, there exists a canonical Dirac structure on $\Sigma$ for which $\bs : (\Sigma,L) \to (M,L_M)$ is forward Dirac -- namely, the \emph{basic} Dirac structure $L=\bs^!(L_M)$:

\begin{definition}
Let $\bs : \Sigma \to M$ be a surjective submersion. A Dirac structure $L$ on $\Sigma$ is called {\bf basic} if there exists a Dirac structure $L_M$ on $M$ such that $L=\bs^!(L_M)$, i.e., for which
\[
 \bs : (\Sigma,L) \rmap (M,L_M)
\]
is a backward Dirac submersion.
\end{definition}
The following proposition describes those Dirac structures on $\Sigma$ which are basic:

\begin{proposition}[Basic criterion]\label{prop: Libermann simple}
Let $\bs:\Sigma\to M$ be a surjective submersion with connected fibres, and denote by $V:=\ker \bs_*\subset T\Sigma$. A Dirac structure $L$ on $\Sigma$ is basic, $L=\bs^!(L_M)$, if and only if $V\subset L$. In this case, we also have that $\bs:(\Sigma,L)\to (M,L_M)$ is a forward Dirac map.
\end{proposition}
\begin{proof}
If $L$ is basic, $L=\bs^!(L_M)$, then clearly $V\subset L$ and, since $\bs$ is a submersion, we also have that $\bs_!\bs^!(L_M)=L_M$.

Conversely, assume that $V\subset L$. Then the flow of vector fields in $V$ preserves $L$. Since the fibres of $\bs$ are connected, this implies that, for every $p,q\in \bs^{-1}(x)$, we can find a diffeomorphism $\varphi:\Sigma\to \Sigma$ such that $\varphi(p)=q$, $\varphi_!(L_p)=L_q$, and which is vertical: $\bs\circ\varphi=\bs$. So $\bs_!(L_q)=\bs_!\varphi_!(L_p)=\bs_!(L_p)$. Thus, there is a well-defined Lagrangian family $x\mapsto L_{M,x}\in \TT_xM$ such that $\bs_!(L_p)=L_{M,\bs(p)}$ for all $p\in \Sigma$.

Smoothness and involutivity of $L_M$ (in the sense of \S 6 and \S 7) are proven as follows. First, remark that the submersion $\bs$ admits local sections $\sigma:U\to \Sigma$, $U \subset M$, and that any such local section $\sigma$ is transverse to $L$, i.e.\ it satisfies:
\[\sigma_{*}(T_xM)+\mathrm{pr}_{T}(L_{\sigma(x)})=T_{\sigma(x)}\Sigma, \ \ \forall \ x\in U,\]
and this condition implies that $\sigma^!(L)$ is a Dirac structure on $U$ (see \S 12). Second, note that $V\subset L$ implies that $L\subset T\Sigma\oplus \im \bs^*$. Now, if $v+\sigma^{*}\alpha\in \sigma^!(L)$, then $\sigma_*v+\alpha\in L\subset T\Sigma\oplus \im \bs^*$, and so $\alpha=\bs^*\beta$ for some $\beta\in T^*M$; hence $\bs_*\sigma_*v+\beta=v+\beta\in L_M$. Thus, $\sigma^!(L)\subset L_M|_U$, and since these spaces have equal dimension, we conclude that $L_M|_U=\sigma^!(L)$. This proves that $L_M$ is a Dirac structure.

Let us conclude by showing that $\bs^!(L_M)=L$. An element in $\bs^!(L_M)$ has the form $v+\bs^*\alpha$, where $\bs_*v+\alpha\in L_M$. Since $L_M=\bs_!(L)$, there exists $w\in T\Sigma$ such that $\bs_*w=\bs_*v$ and $w+\bs^*\alpha\in L$. But then $w-v\in V\subset L$; hence $v+\bs^*\alpha=(w+\bs^*\alpha)+(v-w)\in L$.
\end{proof}

Given a surjective submersion $\bs : \Sigma \to M$ from a Dirac manifold $(\Sigma,L)$, there is a canonical Lagrangian family (see \S 5)
\[
L^{\bs} \subset \TT\Sigma, \quad L^{\bs}_p:=\bs^!(\bs_!(L_p)) \subset \TT_p \Sigma, \quad p \in \Sigma,
\]
which one can view as the Lagrangian family closest to $L$ among Lagrangian families which contain $V:=\ker \bs_*$ \cite[Section 3]{MCF}, since it satisfies
\begin{equation}\label{eq:L^bs}
L^{\bs} = L \cap  V^{\perp} + V.
\end{equation}
Whether or not the Lagrangian family $L^{\bs}$ is a Dirac structure (see \S 6 - \S 8) plays a key role in the following Dirac-geometric version of Libermann's theorem.

\begin{proposition}[Dirac-Libermann]\label{prop: Libermann}
Let $\bs:\Sigma\to M$ be a surjective submersion with connected fibres. A Dirac structure $L$ on $\Sigma$ can be pushed forward via $\bs$ to a Dirac structure on $M$ if and only if the Lagrangian family:
\[L^{\bs}=\bs^!\bs_!(L)\subset \TT\Sigma\]
is a Dirac structure on $\Sigma$, i.e., if $L^{\bs}$ is a smooth, involutive subbundle of $\TT \Sigma$.
\end{proposition}
\begin{proof}
First, if $L$ can be pushed forward to a Dirac structure $L_M$ on $M$, then $L^{\bs}=\bs^!(L_M)$ is a Dirac structure on $\Sigma$, as it is the pullback of a Dirac structure through a surjective submersion (see \S 10).

Conversely, assume that $L^{\bs}$ is a Dirac structure. Note that $\bs_!(L^{\bs}_p)=\bs_!(L_p)$ for all $p\in \Sigma$ (see \S 11); hence it suffices to check that $L^{\bs}$ can be pushed forward. Since $V\subset L^{\bs}$, this follows from Proposition \ref{prop: Libermann simple}.
\end{proof}

The purpose of the next two examples is to highlight that, under the hypotheses of Proposition \ref{prop: Libermann}, neither smoothness nor involutivity of $L^{\bs}$ is ensured.

\begin{example}\label{example : co-dirac product is not smooth}
Consider $\bs:\R^2\to\R$, $\bs(x,y)=x$, and let $L=\mathrm{Gr}(x \dd x \wedge \dd y)$. Then the Lagrangian family $L^{\bs}$ is not smooth; explicitly, it is given by
\[
 L^{\bs}_{(x,y)} = \begin{cases}
\mathbb{R}\partial_y+\mathbb{R}x\dd x & \text{if} \ x \neq 0,\\
\mathbb{R}\partial_x+\mathbb{R}\partial_y & \text{if} \ x = 0.
           \end{cases}
\]
In particular, $\bs$ does not push $L$ forward to a Dirac structure on $\R$.
\end{example}

\begin{example}\label{example : co-dirac product is smooth, not involutive}
Consider $\bs:\R^3\to\R^2$, $\bs(x,y,z)=(x,y)$, and let $L=\mathrm{Gr}(z \partial_x \wedge \partial_y)$. In this case the smooth sections $s_i \in \Gamma(\TT \R^3)$ given by
\[s_1 = z\partial_x-\dd y, \quad s_2=z\partial_y+\dd x, \quad s_3=\partial_z \]
span $L^{\bs}$; hence $L^{\bs}$ is a smooth bundle. However, it is not involutive; e.g., $[s_3,s_1] =\partial_x \notin \Gamma(L^{\bs})$. In particular, $\bs$ does not push $L$ forward to a Dirac structure on $\R^2$.
\end{example}

\subsection*{Sufficient criteria}
Next, we give sufficient criteria for $L^{\bs}$ to be a Dirac structure.

\begin{proposition}[Bundle criterion]\label{prop: bundle criteria}
Let $(\Sigma,L)$ be a Dirac manifold and let $\bs : \Sigma \to M$ be a surjective submersion.
\begin{enumerate}[i)]
\item If a vector subbundle $E \subset \TT \Sigma$ exists, such that
\begin{equation*}\label{assumption}
L^{\bs}=E+V,
\end{equation*}
then $L^{\bs}$ is smooth. This condition is also necessary for $L^{\bs}$ to be smooth.
\item If in addition $E$ is involutive, then also $L^{\bs}$ is involutive. Hence, if $\bs$ has connected fibres, then there is a Dirac structure $L_M$ on $M$ for which $\bs:(\Sigma,L)\to (M,L_M)$ is a forward Dirac map.
\end{enumerate}
\end{proposition}
\begin{proof}
Note that $L^{\bs}$ is the image of the vector bundle map
\[F:E\oplus V\rmap \TT\Sigma,\ \ \ F(e,v):= e+v.\]
Since $L^{\bs}$ is a family of Lagrangian subspaces, we conclude that $F$ has constant rank, and therefore $L^{\bs}$ is smooth. Conversely, if $L^{\bs}$ is smooth, then one can take $E$ to be any smooth complement of $V$ in $L^{\bs}$. This proves i).

By choosing a smooth splitting of the vector bundle map $F$, note that any section $s\in \Gamma(L^{\bs})$ can be represented as $s=v+e$, where $v\in \Gamma(V)$ and $e\in \Gamma(E)$. Assume now that $E$ is involutive. Recall that the smooth Lagrangian subbundle $L^{\bs}=E+V$ is Dirac iff the Courant tensor $\Upsilon\in\Gamma(\bigwedge^3 (L^{\bs})^*\big)$, $\Upsilon(s_1,s_2,s_3)=\langle[s_1,s_2],s_3\rangle$ vanishes identically (see \S 9). So to check that $\Upsilon=0$, it suffices to show that it vanishes on sections which are either in $ \Gamma(V)$ or in $\Gamma(E)$. For $i=1,2,3$, let $v_i\in \Gamma(V)$ and $e_i\in \Gamma(E)$. Because $V$ and $E$ are involutive and $E+V$ is isotropic, we have
\[\Upsilon(v_1,v_2,v_3)=0, \quad \Upsilon(v_1,v_2,e_1)=0, \quad \Upsilon(e_1,e_2,v_3)=0, \quad \Upsilon(e_1,e_2,e_3)=0.\]
Thus, $\Upsilon=0$, and therefore $L^{\bs}$ is Dirac. Hence if $\bs$ has connected fibres, we can invoke Proposition \ref{prop: Libermann} to conclude that $\bs$ pushes $L$ forward. This proves ii).
\end{proof}

Since $L^{\bs}=L\cap V^{\perp}+V$ (\ref{eq:L^bs}), a candidate for the vector bundle $E$ in Proposition \ref{prop: bundle criteria} is the isotropic family $L\cap V^{\perp}$ -- which need not be smooth. The next criterion shows that, if $L^{\bs}$ is smooth, then its involutivity is controlled by that of $L\cap V^{\perp}$.

\begin{proposition}[$L\cap V^{\perp}$ criterion]\label{pro: L[s]}
Let $(\Sigma,L)$ be a Dirac manifold and let $\bs : \Sigma \to M$ be a surjective submersion.
\begin{enumerate}[i)]
\item If $L\cap V^{\perp}$ (or equivalently, if $L\cap V$) is smooth, then $L^{\bs}$ is smooth.
\item If $L^{\bs}$ is smooth, then $L^{\bs}$ is involutive if and only if $L\cap V^{\perp}$ is involutive in the sense of \S 7.
\end{enumerate}
\end{proposition}
\begin{proof}
Note that $L \cap V$ has constant rank iff $L+V$ has constant rank iff $(L+V)^{\perp}=L\cap V^{\perp}$ has constant rank. Thus i) follows from Proposition \ref{prop: bundle criteria} i).

Note that $L\cap V^{\perp}=L \cap L^{\bs}$. Thus, if $L^{\bs}$ is involutive, then, being the intersection of two involutive subbundles, so is $L\cap V^{\perp}$. Conversely, suppose that $L\cap V^{\perp}$ is involutive. Note that the isotropic family $L\cap V^{\perp}$ restricts on an open, dense subset $U \subset \Sigma$ to a vector bundle $E=L\cap V^{\perp}|_U$; hence Proposition \ref{prop: bundle criteria} ii) applies, showing that $L^{\bs}|_U $ is a Dirac structure on $U$. Hence the Courant tensor $\Upsilon$ of $L^{\bs}$ (see \S 9) vanishes on an open, dense subset, whence it vanishes on the whole $\Sigma$. Thus, $L^{\bs}$ is a Dirac structure.
\end{proof}

\begin{remark}\label{rem : Ls basic}
When $L^{\bs}$ is a Dirac structure, and $\bs$ has connected fibres, $L^{\bs}$ is a basic Dirac structure, i.e., $L^{\bs}=\bs^!(L_M)$ for a Dirac structure $L_M$ on $M$. In this case, there is an induced bundle map $\bs_! : L^{\bs} \to L_M$ covering $\bs:\Sigma \to M$, where $s_!(a) = b$ iff $a \in L^{\bs}_{p}$ and $b \in L_{M,\bs(p)}$ are $\bs$-related: $a \sim_{\bs} b$ (see \S 3). Note that this gives rise to an exact sequence of Lie algebroids
\[ 0 \rmap  V \rmap  L^{\bs} \stackrel{\bs_{!}}{\rmap} L_M \rmap  0.\]
The next criterion concerns the existence of a splitting of this sequence.
\end{remark}

\begin{proposition}[Comorphism criterion]\label{prop: comorphism criterion, Dirac}
Let $\bs : \Sigma \to M$ be a surjective submersion with connected fibres from the Dirac manifold $(\Sigma,L)$. If an involutive subbundle $E \subset \TT\Sigma$ exists such that
\begin{equation*}
L^{\bs}=E \oplus V,
\end{equation*}
then there is an induced comorphism of Lie algebroids $\Phi:\bs^*(L_M) \to E$, where $L_M$ is the push-forward Dirac structure on $M$, $L_M=\bs_{!}(L)$.
\end{proposition}
\begin{proof}
By Proposition \ref{prop: bundle criteria} ii), $\bs$ pushes $L$ forward to a Dirac structure $L_M$ on $M$. The surjective vector bundle map $\bs_!: L^{\bs} \to L_M$ covering $\bs$ induces an exact sequence of vector bundles over $\Sigma$
\[
 0 \rmap V \rmap L^{\bs} \rmap \bs^*(L_M) \rmap 0
\]which restricts to a vector bundle isomorphism on $E \diffto \bs^*(L_M)$ (see Remark \ref{rem : Ls basic}). Denote by $\Phi : \bs^*(L_M) \to E$ the inverse isomorphism, which is a vector bundle comorphism. By definition of $\Phi$, we have that, for each $a \in \Gamma(L_M)$, $\Phi^{\dagger}(a) \in \Gamma(E)$ is the unique section such that $\Phi^{\dagger}(a) \sim_{\bs} a$; in particular, it is compatible with anchors. Moreover (see \S 3)
$\Phi^{\dagger}([a,b]) \sim_{\bs} [a,b]$, and $[\Phi^{\dagger}(a),\Phi^{\dagger}(b)] \sim_{\bs} [a,b]$ imply that $\Phi^{\dagger}([a,b])=[\Phi^{\dagger}(a),\Phi^{\dagger}(b)]$; hence $\Phi:\bs^*(L_M) \to E$ is a comorphism of Lie algebroids.
\end{proof}

\subsection*{Coupling Dirac structures}
As an application of Proposition \ref{prop: Libermann}, we consider the case of when a surjective submersion $\bs : \Sigma \to M$ pushes forward a \emph{coupling Dirac structure} $L \subset \TT \Sigma$, i.e., a Dirac structure satisfying:
\[L\cap (V\oplus V^{\circ})=0, \quad \text{where} \ V=\ker\bs_*,\  V^{\circ}=\mathrm{im}\phantom{l}\bs^*.\]
Such Dirac structures are also called \emph{horizontally nondegenerate} \cite{Wade}.

A Lagrangian subbundle $L \subset \TT\Sigma$ satisfying the condition above can be described by a \emph{Vorobjev triple} $(H,\omega,\pi)$, where $H\subset T\Sigma$ is an Ehresmann connection for $\bs:\Sigma \to M$, $\omega$ is a two-form on $H$, and $\pi$ is a vertical bivector field:
\[T\Sigma=H\oplus V, \ \ \omega\in \Gamma(\wedge^2 V^{\circ})\subset \Omega^2(\Sigma), \ \ \pi\in \Gamma(\wedge^2 V )\subset \mathfrak{X}^2(\Sigma),\]
in terms of which $L$ has the direct sum decomposition:
\[L=\{v+\iota_v\omega\ | \ v\in H\}\oplus\{\xi+\pi^{\sharp}\xi \ | \ \xi\in H^{\circ}\}.\]
Involutivity of $L$ is equivalent to the following conditions (see e.g.\ \cite{Wade}):
\begin{enumerate}[(a)]
\item $\pi$ is a Poisson structure: $[\pi,\pi]=0$,
\item $\Lie_{u_1}\pi\in \Gamma(H\wedge T\Sigma)$,
\item $[u_1,u_2]+\pi^{\sharp}\iota_{u_1}\iota_{u_2}\dd \omega\in \Gamma(H)$,
\item $\dd \omega(u_1,u_2,u_3)=0$,
\end{enumerate}
for all $u_1,u_2,u_3\in \Gamma(H)$.

In this situation, Proposition \ref{prop: Libermann} specializes to:
\begin{corollary}[Coupling Dirac structures]\label{coro : vorobjev}
Let $\bs:\Sigma\to M$ be a surjective submersion with connected fibres, and let $L$ be a coupling Dirac structure on $\Sigma$, with corresponding Vorobjev triple $(H,\omega,\pi)$. Then $L$ can be pushed forward via $\bs$ to a Dirac structure on $M$ if and only if $\omega$ is closed. In this case, $\omega=\bs^*\eta$, where $\eta$ is a closed two-form on $M$, $\bs_{!}(L)=\mathrm{Gr}(\eta)$, and moreover, $H$ is involutive.
\end{corollary}
\begin{proof}
It is easy to check that $L^{\bs}=\mathrm{Gr}(\omega)$. This is always a smooth bundle, and it is a Dirac structure iff $\omega$ is closed. Thus, the first part follows from Proposition \ref{prop: Libermann}. By Proposition \ref{prop: Libermann simple}, $\mathrm{Gr}(\omega)$ is a basic Dirac structure, and hence $\omega$ is a basic two-form, i.e.\ $\omega=\bs^*\eta$, where $\eta$ is a closed two-form on $M$, and $\bs:(\Sigma,L)\to (M,\mathrm{Gr}(\eta))$ is a forward Dirac map. By (c) above, $\dd \omega=0$ implies that $H$ is involutive.
\end{proof}

\section{Weak dual pairs and dual pairs}\label{sec : weak dual pairs}

In this section we discuss the notions of \emph{weak dual pairs} and \emph{dual pairs} in Dirac geometry. In the setting of Lie groupoids endowed with multiplicative two-forms, these notions correspond to \emph{over-presymplectic groupoids} \cite[Definition 4.6]{BCWZ} and
\emph{presymplectic groupoids} \cite[Definition 2.1]{BCWZ}, respectively.

\subsection*{Pushing forward two-forms}
In order to motivate the notion of weak dual pairs, we begin by specializing Proposition \ref{prop: Libermann} to the case of closed two-forms. Given a closed two-form $\omega\in \Omega^2(\Sigma)$, denote the corresponding \emph{gauge transformation} by
\begin{align*}
\mathcal{R}_{\omega}:\TT \Sigma\rmap\TT \Sigma, \quad \mathcal{R}_{\omega}(u+\xi):=u+\xi+\iota_u\omega.
\end{align*}
The $\omega$-\emph{orthogonal} of a linear space $E\subset T\Sigma$ will be denoted by
\[E^{\omega}:=T\Sigma\cap \mathcal{R}_{\omega}(E)^{\perp}\subset T\Sigma.\]
Given a surjective submersion $\bs:\Sigma\to M$ and denoting $V:=\ker\bs_*$, we have that
\[\mathrm{Gr}(\omega)\cap V^{\perp}=\mathcal{R}_{\omega}(V^{\omega}),\]
and so the Lagrangian family $\mathrm{Gr}(\omega)^{\bs}$ becomes
\begin{equation}\label{eq: lagr fam}
\mathrm{Gr}(\omega)^{\bs}=V+\mathrm{Gr}(\omega) \cap V^{\perp}=V+\mathcal{R}_{\omega}(V^{\omega}).
\end{equation}
The criteria from Section \ref{sec : Pushing forward Dirac structures} specialize to the following:

\begin{corollary}\label{cor : forms}
Let  $\bs : \Sigma \to M$ be a surjective submersion with connected fibres, $\omega \in \Omega^2(\Sigma)$ be a closed two-form, and let $V:=\ker \bs_*$.
\begin{enumerate}[i)]
\item $\bs$ pushes $\mathrm{Gr}(\omega)$ to a Dirac structure $L_M$ on $M$ if and only if the Lagrangian family $\mathrm{Gr}(\omega)^{\bs}=V+\mathcal{R}_{\omega}(V^{\omega})$ is a Dirac structure on $\Sigma$. In this case,
\[L_M \ \ \text{is Poisson} \quad \iff \quad \ker\omega\subset V.\]
\item If $V^{\omega}$ is smooth then also $\mathrm{Gr}(\omega)^{\bs}$ is smooth. On the other hand, if $\mathrm{Gr}(\omega)^{\bs}$ is smooth then it is a Dirac structure exactly when the linear family $V^{\omega}$ is involutive.
\item If there exists a smooth subbundle $W\subset T\Sigma$ such that
\begin{equation}\label{eq : important}
\mathrm{Gr}(\omega)^{\bs}=V+\mathcal{R}_{\omega}(W),
\end{equation}
then $\mathrm{Gr}(\omega)^{\bs}$ is smooth. If $W$ is also involutive, then $\mathrm{Gr}(\omega)^{\bs}$ is a Dirac structure, and so $\bs$ pushes $\mathrm{Gr}(\omega)$ forward to a Dirac structure $L_M$ on $M$.

\noindent Moreover, if the decomposition above is a direct sum
\begin{equation*}
\mathrm{Gr}(\omega)^{\bs}=V\oplus \mathcal{R}_{\omega}(W)
\end{equation*}
then there is an induced Lie algebroid comorphism $\Phi:\bs^*(L_M) \to W$.
\end{enumerate}
\end{corollary}

\begin{proof}
Proposition \ref{prop: Libermann} implies the first part of $i)$; for the second part,
note that
\[\bs_!(L_{p})\cap T_{\bs(p)}M=\bs_*(\ker\omega_p),\]
and therefore the induced Dirac structure is Poisson if and only if this space is trivial, which is equivalent to $\ker\omega\subset V$. Item $ii)$ is implied by Proposition \ref{pro: L[s]}. The first part of item $iii)$ follows from Propositions \ref{prop: bundle criteria} and the second from Proposition \ref{prop: comorphism criterion, Dirac}.
\end{proof}

Thus the condition that there be an involutive $W \subset T\Sigma$ satisfying (\ref{eq : important}) is sufficient to ensure that a surjective submersion $\bs:\Sigma \to M$ pushes $\mathrm{Gr}(\omega)$ forward, but it is certainly not necessary, as the following example illustrates:

\begin{example}\label{ex : assumption not satisfied}
Consider $\bs:\R^3\to \R$, $\bs(x,y,z)=x$ and $\omega=\dd (x^2 y)\wedge \dd z$. Then 
\[\bs:(\R^3,\mathrm{Gr}(\omega))\rmap (\R,T\R)\]
is a forward Dirac map.

Note that there cannot be a smooth subbundle $W \subset T\Sigma$ satisfying $\mathrm{Gr}(\omega)^{\bs}=V+\mathcal{R}_{\omega}(W)$. Indeed, such a subbundle would have to be spanned by the vector field $v:=x\partial_x-2y\partial_y$ on $x \neq 0$, yet this cannot be extended smoothly (as a line bundle) over points of the form $(0,0,z)$, since:
\[\lim_{y\to 0}\lim_{x\to 0}\langle v\rangle=\langle\partial_y\rangle\ \ \textrm{and} \ \ \lim_{x\to 0}\lim_{y\to 0}\langle v\rangle=\langle\partial_x\rangle.\]
\end{example}

\begin{remark}\label{rem : V perp not smooth}
Part ii) of Corollary \ref{cor : forms} resembles the most Libermann's theorem in Poisson geometry. Namely, assuming that $\mathrm{Gr}(\omega)^{\bs}$ is smooth, $\bs$ pushes $\mathrm{Gr}(\omega)$ forward if and only if $V^{\omega}$ is involutive. However, note that even if this is the case, the rank of $V^{\omega}$ need not be lower semicontinuous (see Example \ref{ex : assumption not satisfied}); thus, even if it is involutive, it need not be a singular foliation (in the sense of \cite{Stefan}).
\end{remark}

\subsection*{Weak dual pairs}

The key ingredient in defining weak dual pairs is condition (\ref{eq : important}) of Corollary \ref{cor : forms} $iii)$. First, we give several algebraic reformulations of this condition, which, in particular, show that $V$ and $W$ play a symmetric role (note that condition (a) in Lemma \ref{lem : equivalent weak} below is the pointwise version of (\ref{eq : important})).

\begin{lemma}\label{lem : equivalent weak}
Consider a finite dimensional vector space $A$, vector subspaces $B\subset A$ and $C\subset A$, and a 2-form $\omega\in \wedge^2A^*$. Denote $\ker \omega\subset A$ by $K$. Then the following are equivalent:
\begin{enumerate}[$(a)$]
\item $B+\mathcal{R}_{\omega}(B^{\omega})=B+\mathcal{R}_{\omega}(C)\subset A\oplus A^*$;
\item $C+\mathcal{R}_{-\omega}(C^{\omega})=C+\mathcal{R}_{-\omega}(B)\subset A\oplus A^*$;
\item $B^{\omega}=C+B\cap K$;
\item $C^{\omega}=B+C\cap K$;
\item $\omega(B,C)=0$ and $\mathrm{dim}(B\cap K \cap C)=\mathrm{dim}(B)+\mathrm{dim}(C)-\mathrm{dim}(A)$.
\end{enumerate}
\end{lemma}
\begin{proof}
Note first that all conditions imply $\omega(B,C)=0$. Next, note that:
\[\mathrm{dim}(B^{\omega})=\dim(A)-\mathrm{dim}(B)+\mathrm{dim}(B\cap K),\]
\[\mathrm{dim}(C+B\cap K)=\mathrm{dim}(C)+\mathrm{dim}(B\cap K)-\mathrm{dim}(B\cap K\cap C).\]
Clearly, $C+B\cap K\subset B^{\omega}$; thus the two are equal iff they have the same dimension. This shows that $(c)\Leftrightarrow(e)$. Similarly, $(d)\Leftrightarrow(e)$.

On the other hand, $B+\mathcal{R}_{\omega}(C)\subset B+\mathcal{R}_{\omega}(B^{\omega})$. The latter is a Lagrangian subspace of $A\oplus A^*$, thus $(a)$ holds iff
\[\mathrm{dim}(B+\mathcal{R}_{\omega}(C))=\mathrm{dim}(A),\]
which is equivalent to $(e)$, because
\[\mathrm{dim}(B+\mathcal{R}_{\omega}(C))=\mathrm{dim}(B)+\mathrm{dim}(C)-\mathrm{dim}(B\cap K\cap C).\]
Hence $(a)\Leftrightarrow (e)$. Similarly, $(b)\Leftrightarrow (e)$.
\end{proof}

Motivated by Corollary \ref{cor : forms} $iii)$ and Lemma \ref{lem : equivalent weak}, we introduce:
\begin{definition}\label{def : weak dual pair}
A {\bf weak dual pair} consists of surjective, forward Dirac submersions
\begin{align}\label{eq:wdp}
(M_0,L_0)\stackrel{\bs}{\lmap} (\Sigma,\mathrm{Gr}(\omega)) \stackrel{\bt}{\rmap} (M_1,-L_{1}),
\end{align}
where $\omega$ is a closed two-form on $\Sigma$, satisfying
\[\tag{$*$}\omega(V,W)=0\quad \textrm{and}\quad \mathrm{rank}(V\cap K\cap W)=\mathrm{dim}(\Sigma)-\mathrm{dim}(M_0)-\mathrm{dim}(M_1),\]
where $V:=\ker\bs_*$, $W:=\ker \bt_*$, and $K:=\ker\omega$.
\end{definition}

A first consequence of this definition is that
\begin{lemma}\label{lem : VWK}
For a weak dual pair, we have that $V\cap K\cap W$ is a smooth, involutive subbundle of $T\Sigma$, where we used the notation of Definition \ref{def : weak dual pair}.
\end{lemma}
\begin{proof}
Note that the kernel of the vector bundle map
\[V\oplus W\rmap \TT\Sigma, \ \ v\oplus w\mapsto v+\mathcal{R}_{\omega}(w)\]
can be identified with $V\cap K\cap W$. This map has constant rank, because its image is $V+\mathcal{R}_{\omega}(W)$, which, by Lemma \ref{lem : equivalent weak} (a) is Lagrangian. Therefore $V\cap K\cap W$ is smooth. The bundles $V$ and $W$ are involutive, and since $\omega$ is closed, also $K$ is involutive in the sense of \S 7. Thus $V\cap K\cap W$ is involutive.
\end{proof}

By Lemma \ref{lem : equivalent weak}, condition ($*$) in Definition \ref{def : weak dual pair} is the same as condition (\ref{eq : important}) in Corollary \ref{cor : forms}. Due to the symmetric role played by $V$ and $W$, we may apply Corollary \ref{cor : forms} $iii)$ and obtain an intrinsic description of weak dual pairs:
\begin{corollary}
Let $\omega \in \Omega^2(\Sigma)$ be a closed two-form, and let $\bs:\Sigma\to M_0$ and $\bt:\Sigma\to M_1$ be surjective submersions with connected fibres. If $V:=\ker\bs_*$ and $W:=\ker\bt_*$ satisfy $(*)$, then $M_0$ and $M_1$ carry Dirac structures $L_0$ and $L_1$ respectively, which yield a weak dual pair
\[(M_0,L_0)\stackrel{\bs}{\lmap} (\Sigma,\mathrm{Gr}(\omega)) \stackrel{\bt}{\rmap} (M_1,-L_{1}).\]
\end{corollary}

Next we give some alternative descriptions of weak dual pairs.

\begin{proposition}\label{pro : equivalent wdp}
Let $\omega \in \Omega^2(\Sigma)$ be a closed two-form, $(M_0,L_0)$ and $(M_1,L_1)$ be Dirac manifolds and $\bs:\Sigma\to M_0$ and $\bt:\Sigma\to M_1$ be surjective submersions. The following are equivalent:
\begin{enumerate}[i)]
\item the diagram below is a weak dual pair:
\[(M_0,L_0)\stackrel{\bs}{\lmap} (\Sigma,\mathrm{Gr}(\omega)) \stackrel{\bt}{\rmap} (M_1,-L_{1});\]
\item $\bs^!(L_0) = \mathcal{R}_{\omega}(\bt^!(L_1))$ and $\mathrm{rank}(V\cap K\cap W)=\mathrm{dim}(\Sigma)-\mathrm{dim}(M_0)-\mathrm{dim}(M_1)$;
 \item $(\bs,\bt):(\Sigma,\mathrm{Gr}(\omega)) \to (M_0,L_0) \times (M_1,-L_1)$ is a forward Dirac map and either $\mathrm{rank}(V\cap K\cap W)=\mathrm{dim}(\Sigma)-\mathrm{dim}(M_0)-\mathrm{dim}(M_1)$ or $\omega(V,W)=0$ holds;
\item $\bs^!(L_0) = \mathcal{R}_{\omega}(\bt^!(L_1))$ and $(\bs,\bt):(\Sigma,\mathrm{Gr}(\omega)) \to (M_0,L_0) \times (M_1,-L_1)$ is a forward Dirac map,
\end{enumerate}
where $V:=\ker\bs_*$, $W:=\ker\bt_*$, and $K:=\ker\omega$.
\end{proposition}

\begin{proof}
First, we claim that $(\bs,\bt)$ being forward Dirac is equivalent to the inclusions:
\begin{equation}\label{F2}
\bs^!(L_0) \subset V+\mathcal{R}_{\omega}(W), \quad \bt^!(L_1) \subset \mathcal{R}_{-\omega}(V)+W.
\end{equation}
Assume that $(\bs,\bt)$ is a forward Dirac map. Let $\widetilde{u}_0+\bs^*(\xi_0) \in \bs^!(L_0)_p$ be $\bs$-related to $u_0+\xi_0 \in L_{0,\bs(p)}$. By assumption, for all $(u_0+\xi_0,u_1-\xi_1) \in L_{0,\bs(p)} \times -L_{1,\bt(p)}$, there is $u \in T_p\Sigma$ such that $u+\iota_u\omega \sim_{(\bs,\bt)} (u_0+\xi_0,u_1-\xi_1)$. Applying this to $(u_0+\xi_0,0)$, we deduce the existence of $w \in W_p$ such that $w+\iota_w\omega \sim_{\bs} u_0+\xi_0$; hence
\[
 \widetilde{u}_0 = w+v, \quad \iota_w\omega = \bs^*(\xi_0)
\]
for some $v \in V_p$. Therefore
\[
\widetilde{u}_0+\bs^*(\xi_0) = v+w+\iota_w\omega \in V+\mathcal{R}_{\omega}(W),\]
which shows that $\bs^!(L_0) \subset V+\mathcal{R}_{\omega}(W)$. The inclusion $\bt^!(L_1) \subset \mathcal{R}_{-\omega}(V)+W$ is proved similarly. Hence (\ref{F2}) holds.

Conversely, assume that (\ref{F2}) holds. Then for $(u_0+\xi_0,u_1-\xi_1) \in L_{0,\bs(p)} \times -L_{1,\bt(p)}$ there exist $v_0,v_1 \in V_p$ and $w_0,w_1 \in W_p$, such that
\[
 v_0+w_0+\iota_{w_0}\omega \sim_{\bs} u_0+\xi_0, \quad v_1+w_1+\iota_{v_1}\omega \sim_{\bt} u_1-\xi_1;
\]this implies that
\[
 w_0+\iota_{w_0}\omega \sim_{\bs} u_0+\xi_0, \quad v_1+\iota_{v_1}\omega \sim_{\bt} u_1-\xi_1,
\]
and so $v_1+w_0+\iota_{v_1+w_0}\omega \sim_{(\bs,\bt)} (u_0+\xi_0,u_1-\xi_1)$; hence $(\bs,\bt)$ is forward Dirac.

Assume now that $i)$ holds. Then by Lemma \ref{lem : equivalent weak}, the following equalities below
\[\bs^!(L_0) =(\mathrm{Gr}(\omega))^{\bs} =  V+\mathcal{R}_{\omega}(W), \quad \bt^!(L_1) = (\mathrm{Gr}(\omega))^{\bt}= \mathcal{R}_{-\omega}(V)+W,
\]hold, and imply both the condition $\bs^!(L_0) = \mathcal{R}_{\omega}(\bt^!(L_1))$ and the inclusions (\ref{F2}). Since condition $(*)$ is implied by $i)$, we conclude that $i)$ implies both $ii)$ and $iii)$.

Assume now that $ii)$ holds. Since $V\subset \bs^!(L_0)$ and $W\subset \bt^{!}(L_1)$, $ii)$ yields
\[V+\mathcal{R}_{\omega}(W) \subset \bs^!(L_0)=\mathcal{R}_{\omega}(\bt^!(L_1)).\]
The rank condition in $ii)$ implies that $\mathrm{rank}(V+\mathcal{R}_{\omega}(W))=\mathrm{dim}(\Sigma)$, and so the above must hold with equality
\[V+\mathcal{R}_{\omega}(W) = \bs^!(L_0)=\mathcal{R}_{\omega}(\bt^!(L_1)).\]The fact that this space is isotropic gives $\omega(V,W)=0$, hence $V+\mathcal{R}_{\omega}(W)\subset V+\mathcal{R}_{\omega}(V^{\omega})\subset\mathrm{Gr}(\omega)^{\bs}$. Again, since both spaces are Lagrangian, we conclude that
\[V+\mathcal{R}_{\omega}(W)=\mathrm{Gr}(\omega)^{\bs}=\bs^!(L_0).\]
This implies that $\bs$ is forward Dirac, and similarly, one shows that $\bt$ is forward Dirac. Thus $ii)$ implies $i)$.

Assume now that $iii)$ holds. By the claim, (\ref{F2}) holds. The condition $\omega(V,W)=0$ implies that $V+\mathcal{R}_{\omega}(W)$ is isotropic, so $\mathrm{rank}(V+\mathcal{R}_{\omega}(W))\leq \mathrm{dim}(\Sigma)$, and the rank condition from $iii)$ implies that $\mathrm{rank}(V+\mathcal{R}_{\omega}(W))= \mathrm{dim}(\Sigma)$. Assuming that either $\mathrm{rank}(V\cap K\cap W)=\mathrm{dim}(\Sigma)-\mathrm{dim}(M_0)-\mathrm{dim}(M_1)$ or $\omega(V,W)=0$ yields equalities in (\ref{F2}), and so $\bs^!(L_0) = \mathcal{R}_{\omega}(\bt^!(L_1))$ and $\mathrm{rank}(V+\mathcal{R}_{\omega}(W))= \mathrm{dim}(\Sigma)$, which is equivalent to the rank condition in $ii)$. Hence, $iii)$ implies $ii)$.

Clearly, by the discussion above, $i)$ implies $iv)$. Finally, assume that $iv)$ holds. Note that $\bs^!(L_0) = \mathcal{R}_{\omega}(\bt^!(L_1))$ implies that $V+\mathcal{R}_{\omega}(W)\subset \bs^!(L_0)$; hence $V+\mathcal{R}_{\omega}(W)$ is isotropic, and so $\omega(V,W)=0$. Thus $iii)$ holds.
\end{proof}

\begin{remark}\label{remark : pre dual pair condition}
In \cite{BursztynRadko} (see also \cite[Lemma 4.2]{BCWZ}), a different version of weak dual pairs is considered. Namely, the diagram (\ref{eq:wdp}) from Definition \ref{def : weak dual pair} is called a \emph{pre-dual pair} if instead of $(*)$ it satisfies (see \cite[Definition 3.1]{BursztynRadko}):
\begin{align}\tag{$**$}
V^{\omega} = W+K\ \ \textrm{and}\ \ W^{\omega}=V+K.
\end{align}
Note that these two conditions are equivalent: by taking the $\omega$-orthogonal of the first relation one obtains the second. Weak dual pairs satisfy $(**)$ (see (c) in Lemma \ref{lem : equivalent weak}), but not conversely (see Example \ref{example : co-dirac product is not smooth} with $W=V=\ker \bs_*$, and also Example \ref{example : pre-dual pair but not weak dual pair}).

That condition $(**)$ is not sufficient for $\bs$ to push the Dirac structure $\mathrm{Gr}(\omega)$ forward should be contrasted to the setting in \cite[Lemma 4.2]{BCWZ} -- where $(**)$ is enough to ensure that a multiplicative two-form on a Lie groupoid pushes forward via the source map.
\end{remark}

Below are two examples of diagrams of surjective, forward Dirac submersions,
as in \ref{eq:wdp}, which are not weak dual pairs.

\begin{example}\label{example : not pre-dual, (s,t) forward}
Let $\bs:(\R^3,\mathrm{Gr}(\omega))\rmap (\R,T\R)$ be as in Example \ref{ex : assumption not satisfied}, and consider $\bt:(\R^3,\mathrm{Gr}(\omega)) \to (\R^0,0)$, which is trivially a forward map into the trivial Dirac structure on $\R^0$. Then\[(\R,T\R) \stackrel{\bs}{\lmap} (\Sigma,\mathrm{Gr}(\omega)) \stackrel{\bt}{\rmap} (\R^0,0)\]is a diagram of surjective, forward Dirac submersions, in which \[(\bs,\bt):(\Sigma,\mathrm{Gr}(\omega)) \to (\R,T\R) \times (\R^0,0)\]is forward Dirac. However, $\bs^!(L_0) \neq \mathcal{R}_{\omega}(\bt^!(L_1))$. Clearly, the fibres of $\bs$ and $\bt$ are not $\omega$-orthogonal.
\end{example}

\begin{example}\label{example : pre-dual pair but not weak dual pair}
Consider the forward Dirac submersions
 \[(\R^2,L_0) \stackrel{\bs}{\lmap} (\R^3,\mathrm{Gr}(\omega)) \stackrel{\bt}{\rmap} (\R^2,-L_1),\]
 where $\bs(x,y,z):=(x,y)$, $\bt(x,y,z):=(y,z)$, $\omega:= \dd x \wedge \dd y + \dd y \wedge \dd z$, $L_0:=\langle \partial_x, \dd y \rangle$ and $L_1:=\langle \partial_z, \dd y \rangle$. In this case, $\bs^!(L_0) = \mathcal{R}_{\omega}(\bt^!(L_1))=\langle \partial_x, \partial_z, \dd y \rangle$. However, $(\bs,\bt):(\R^3,\mathrm{Gr}(\omega)) \to (\R^2,L_0) \times (\R^2,-L_1)$ is not forward Dirac. Moreover, the fibres of $\bs$ and $\bt$ are $\omega$-orthogonal.

This is also an example of a pre-dual pair (see Remark \ref{remark : pre dual pair condition}) which is not a weak dual pair, i.e.\ condition $(**)$ holds, but condition $(*)$ (Lemma \ref{lem : equivalent weak} $(c)$) fails:
 \[
  V^{\omega} = W + K\neq W + V \cap K. 
 \]
\end{example}

\subsection*{Dual pairs}
Next, we introduce the main notion of this paper:

\begin{definition}\label{def : dual pair}
A \textbf{dual pair} is a weak dual pair
\[(M_0,L_0)\stackrel{\bs}{\lmap} (\Sigma,\mathrm{Gr}(\omega)) \stackrel{\bt}{\rmap} (M_1,-L_{1}),\]
satisfying the condition
\[\tag{${\davidsstar}$} V\cap K\cap W=0,\]
where $V:=\ker\bs_*$, $W:=\ker \bt_*$ and $K:=\ker\omega$. 

Equivalently, by $(*)$, a dual pair is weak dual pair satisfying:
\[\mathrm{dim}(\Sigma)=\mathrm{dim}(M_0)+\mathrm{dim}(M_1).\]
\end{definition}

Analogously to Lemma \ref{lem : equivalent weak}, the algebraic conditions appearing in the definition of dual pairs can be reformulated as follows:
\begin{lemma}\label{lema: equivalent dual}
In the setting of Lemma \ref{lem : equivalent weak}, the following are equivalent
\begin{enumerate}[$(a)$]
\item $B+\mathcal{R}_{\omega}(B^{\omega})=B\oplus \mathcal{R}_{\omega}(C)$;
\item $C+\mathcal{R}_{-\omega}(C^{\omega})=C\oplus\mathcal{R}_{-\omega}(B)$;
\item $B^{\omega}=C\oplus B\cap K$;
\item $C^{\omega}=B\oplus C\cap K$;
\item $\omega(B,C)=0$ and $B\cap K\cap C=0$ and $\mathrm{dim}(A)=\mathrm{dim}(B)+\mathrm{dim}(C)$.
\end{enumerate}
\end{lemma}

Before giving further descriptions of dual pairs, we recall the following notion:

\begin{definition}
A forward Dirac map
\[\varphi : (M_0,L_0) \to (M_1,L_1)\]
is called {\bf strong} if $L_0 \cap \ker \varphi_* = 0$ \cite{ABM}. When $L_0$ is the graph of a closed two-form $\omega_0 \in \Omega^2(M_0)$, we call it a {\bf presymplectic realization} \cite[Definition 7.1]{BCWZ}.
\end{definition}


\begin{proposition}\label{pro : equivalent dual pair}
Consider a diagram of surjective submersions
\begin{align*}M_0 \stackrel{\bs}{\lmap} \Sigma \stackrel{\bt}{\rmap} M_1, \ \ \ \ \ \text{where} \ \  \dim \Sigma = \dim M_0 + \dim M_1.\end{align*}

Let $\omega \in \Omega^2(\Sigma)$ be a closed two-form, and let $L_i \subset \TT M_i$ be Dirac structures. Then the following conditions are equivalent:
\begin{enumerate}[i)]
\item the diagram below is a dual pair:
\[(M_0,L_0)\stackrel{\bs}{\lmap} (\Sigma,\mathrm{Gr}(\omega)) \stackrel{\bt}{\rmap} (M_1,-L_{1});\]
\item $\bs^!(L_0) = \mathcal{R}_{\omega}(\bt^!(L_1))$ and $V\cap K\cap W=0$;
\item $(\bs,\bt):(\Sigma,\mathrm{Gr}(\omega)) \to (M_0,L_0) \times (M_1,-L_1)$ is a presymplectic realization;
\end{enumerate}
where $V:=\ker\bs_*$, $W:=\ker\bt_*$, and $K:=\ker\omega$.
\end{proposition}
\begin{proof}
Note that, under the assumption $\dim \Sigma=\dim M_0+\dim M_1$, the first two items are equivalent to the corresponding items of Proposition \ref{pro : equivalent wdp}. Similarly, item $iii)$ is equivalent to $(\bs,\bt)$ being a forward Dirac map, and to $V\cap W\cap K=0$; thus it is equivalent to the first version of the corresponding item in Proposition \ref{pro : equivalent wdp}. Thus, the result follows from Proposition \ref{pro : equivalent wdp}.
\end{proof}

\begin{example}\label{ex : foliations have presymplectic realizations}
A foliation $\mathscr{F}$ on a smooth manifold $M$ admits a presymplectic realization. Indeed, let $L \subset \TT M$ be the Dirac structure corresponding to $\mathscr{F}$. Denote the conormal bundle of $\mathscr{F}$ by $N^*\mathscr{F} \subset T^*M$, the bundle projection by $\bs : N^*\mathscr{F} \to M$, and by $\omega \in \Omega^2(N^*\mathscr{F})$ the restriction of the canonical symplectic form on $T^*M$. Then
 \[
  \bs : (N^*\mathscr{F},\mathrm{Gr}(\omega)) \rmap (M,L)
 \]is a presymplectic realization.

Assume that the foliation $\mathscr{F}$ is given by the fibres of a surjective submersion $p:M \to B$. Then the above presymplectic realization is part of a dual pair
 \[
  (M,L) \stackrel{\bs}{\lmap} (N^*\mathscr{F},\mathrm{Gr}(\omega)) \stackrel{\bt}{\rmap} (B,T^*B), \quad \bt:=p \circ \bs.
 \]
\end{example}

\begin{example}
If $(M_0,L_0)\stackrel{\bs}{\lmap} (\Sigma,\mathrm{Gr}(\omega)) \stackrel{\bt}{\rmap} (M_1,-L_{1})$ is a dual pair, and $\sigma_0 \in \Omega^2(M_0)$, $\sigma_1 \in \Omega^2(M_1)$ are closed two-forms, then
 \begin{align*}(M_0,\mathcal{R}_{\sigma_0}(L_0))\stackrel{\bs}{\lmap} (\Sigma,\mathrm{Gr}(\omega+\bs^*(\sigma_0)-\bt^*(\sigma_1))) \stackrel{\bt}{\rmap} (M_1,-\mathcal{R}_{\sigma_1}(L_{1}))\end{align*}is again a dual pair, as follows by the description provided in Proposition \ref{pro : equivalent dual pair} $iii)$.
\end{example}

The main property of strong forward Dirac maps is contained in the following result from \cite[Proposition 2.8]{ABM} (see also \cite[Lemma 7.3]{BCWZ}):

\begin{lemma}[Strong forward Dirac maps]\label{lem : Strong forward Dirac maps}
A strong, forward Dirac map $\varphi : (\Sigma,L_{\Sigma}) \to (M,L)$ induces a comorphism of Lie algebroids $\Phi:\varphi^*(L) \to L_{\Sigma}$.

A presymplectic realization $\varphi : (\Sigma,\mathrm{Gr}(\omega)) \to (M,L)$ induces a comorphism of Lie algebroids $\Psi:\varphi^*(L) \to T\Sigma$.
\end{lemma}
\begin{proof}
The forward condition on $\varphi$ means that, for each $x \in \Sigma$ and $a\in L_{\varphi(x)}$, there is at least one $b \in L_{\Sigma,x}$ such that $b \sim_{\varphi} a$, and the strong condition means that at most one such $b$ exists. This defines a comorphism of vector bundles $\Phi:\varphi^*(L) \to L_{\Sigma}$ by setting $b=:\Phi_x(a)$. If now $a_1,a_2 \in \Gamma(L)$, then $\Phi^{\dagger}(a_1) \sim_{\varphi} a_1$ and $\Phi^{\dagger}(a_2) \sim_{\varphi} a_2$; therefore $\Phi^{\dagger}([a_1,a_2]) \sim_{\varphi} [a_1,a_2]$; by uniqueness, it follows that $\Phi^{\dagger}([a_1,a_2])=[\Phi^{\dagger}(a_1),\Phi^{\dagger}(a_2)]$ and hence $\Phi^{\dagger}$ is a homomorphism of Lie algebras. This proves the first assertion, and the second follows by composing $\Phi$ with the anchor map $\mathrm{pr}_T:\mathrm{Gr}(\omega) \to T\Sigma$.
\end{proof}

\begin{remark}\label{rem : bracket relations}
 Let $(M_0,L_0)\stackrel{\bs}{\lmap} (\Sigma,\mathrm{Gr}(\omega)) \stackrel{\bt}{\rmap} (M_1,-L_{1})$ be a dual pair. By Proposition \ref{pro : equivalent dual pair} $iii)$ and Lemma \ref{lem : Strong forward Dirac maps} above, we obtain a comorphism of Lie algebroids
 \[\Psi:(\bs, \bt)^*(L_0\times -L_1)\rmap T\Sigma.\]
 Restricting this to the first and second components, we obtain induced comorphisms of Lie algebroids:
 \[
  \Psi_0 : \bs^*(L_0) \rmap W = \ker\bt_*, \quad \Psi_1 : \bt^*(L_1) \rmap V= \ker\bs_*.
 \]
 For sections $a_i \in \Gamma(L_i)$, $w:=\Psi_0^{\dagger}(a_0) \in \Gamma(W)$ is the unique vector field tangent to the fibres of $\bt$, satisfying $w+\iota_w\omega \sim_{\bs} a_0$; similarly, $v:=\Psi_1^{\dagger}(a_1)$ is the unique vector field tangent to the fibres of $\bs$ satisfying $v-\iota_v\omega \sim_{\bt} a_1$. Note also that any two such vector fields $\Psi_0^{\dagger}(a_0),\Psi_1^{\dagger}(a_1)$ are $\omega$-orthogonal and commute, since
 \[
  [\Psi_0^{\dagger}(a_0),\Psi_1^{\dagger}(a_1)]+\iota_{[\Psi_0^{\dagger}(a_0),\Psi_1^{\dagger}(a_1)]}\omega \sim_{(\bs,\bt)} [(a_0,0),(0,\overline{a_1})] = 0,
 \]where for a section $a=u+\xi$, $\overline{a}$ denotes the section $u-\xi$.
\end{remark}

The main example of dual pairs is the following:

\begin{example}\label{remark : hausdorff integrations}
If a Dirac manifold $(M,L)$ is integrable by a (Hausdorff) presymplectic groupoid $(G,\omega)\rightrightarrows (M,L)$, then
\[
(M,L) \stackrel{\bs}{\lmap} (G,\mathrm{Gr}(\omega)) \stackrel{\bt}{\rmap} (M,-L)
\]
forms a dual pair \cite{BCWZ}. In this case, the induced comorphism of Lie algebroids $\bs^*(L) \to W$ is the canonical action by left-invariant vector fields of the Lie algebroid $L$ on its Lie groupoid $G$, and is automatically complete (in the sense of \S 13). Our Main Theorem is strongly related to this example (see Section \ref{sec : further}).
\end{example}

It was proven in \cite[Theorem 8]{CrFer04} that the existence of a complete symplectic realization of a Poisson manifold implies its integrability; but, as pointed out in \cite[Corollary 7]{CrFer04}, the proof depends only upon the existence of a complete action of the corresponding Lie algebroid. Applying these results in the same way as in the case of presymplectic realizations \cite[Remark 7.5]{BCWZ}, one obtains the following:

\begin{corollary}[Integrability criterion]\label{coro: int}
Consider a surjective, forward Dirac submersion $\bs : (\Sigma,\mathrm{Gr}(\omega)) \to (M,L)$. Assume that there is an involutive subbundle $W\subset T\Sigma$ such that:
\begin{equation*}
\mathrm{Gr}(\omega)^{\bs}=V\oplus \mathcal{R}_{\omega}(W), \ \ \textrm{where}\ V:=\ker s_*.
\end{equation*}

If the induced comorphism of Lie algebroids $\bs^*(L) \to W$ (see Corollary \ref{cor : forms} $iii)$) is complete (see \S 13), then $(M,L)$ is integrable by a (not necessarily Hausdorff) presymplectic groupoid. In that case, the source-simply connected Lie groupoid integrating $(M,L)$ is Hausdorff if and only if the holonomy groupoid $\mathrm{Hol}(W)\rightrightarrows \Sigma$ of the foliation $W$ is Hausdorff.
\end{corollary}

\section{Operations with weak dual pairs}\label{sec : properties}

In this section we discuss operations that can be performed with weak dual pairs: the operation of \emph{composition} of weak dual pairs, the operation of \emph{pullback along a surjective submersion}, and its partial inverse, called the \emph{reduction} procedure (which, under certain assumptions, allows one to reduce a weak dual pair to a dual pair), and finally the \emph{pullback procedure along transverse maps}.

\subsection*{Composition}

Weak dual pairs can be \emph{composed}. Namely, suppose
\begin{align*}(M_0,L_0) \stackrel{\bs_{01}}{\lmap} (\Sigma_{01},\mathrm{Gr}(\omega_{01})) \stackrel{\bt_{01}}{\rmap} (M_1,-L_1), \\ (M_1,L_1) \stackrel{\bs_{12}}{\lmap}  (\Sigma_{12},\mathrm{Gr}(\omega_{12})) \stackrel{\bt_{12}}{\rmap} (M_2,-L_2)\end{align*}are weak dual pairs. Their \emph{composition} is
\begin{align*}
(M_0,L_0) \stackrel{\bs_{02}}{\lmap} & (\Sigma_{02},\mathrm{Gr}(\omega_{02})) \stackrel{\bt_{02}}{\rmap} (M_2,-L_2),
\end{align*}where
\begin{align*}
& \Sigma_{02}:=\Sigma_{01} \times_{M_1} \Sigma_{12}, \quad \omega_{02}:=\mathrm{pr}_1^*(\omega_{01})+\mathrm{pr}_2^*(\omega_{12}),\\
& M_0 \stackrel{\bs_{02}}{\lmap} \Sigma_{02} \stackrel{\bt_{02}}{\rmap} M_2, \quad \bs_{02}:=\bs_{01} \circ \mathrm{pr}_1, \ \ \bt_{02}:=\bt_{12} \circ \mathrm{pr}_2.\end{align*}

Pictorially:
\[
\xymatrix@C-15pt{
&& (\Sigma_{02},\mathrm{Gr}(\omega_{02})) \ar[dl]_{\mathrm{pr}_1} \ar[dr]^{\mathrm{pr}_2} && \\
&  (\Sigma_{01},\mathrm{Gr}(\omega_{01})) \ar[dl]_{\bs_{01}} \ar[dr]^{\bt_{01}} & &  (\Sigma_{12},\mathrm{Gr}(\omega_{12}))  \ar[dl]_{\bs_{12}} \ar[dr]^{\bt_{12}} &\\
(M_0,L_0) & & (M_1,\pm L_1)  & & (M_2, -L_2)
}
\]

\begin{proposition}\label{lem : composition}
The composition of weak dual pairs is again a weak dual pair.
\end{proposition}
\begin{proof}
We must check that
\begin{align*}\label{eq : dual pair?}
(M_0,L_0) \stackrel{\bs_{02}}{\lmap} (\Sigma_{02},\mathrm{Gr}(\omega_{02})) \stackrel{\bt_{02}}{\rmap} (M_2,-L_2)
\end{align*}is again a weak dual pair. First, since $\bs_{12}$ is a submersion, the map $\mathrm{pr}_1:\Sigma_{02}\to \Sigma_{01}$ is a submersion. Similarly, $\mathrm{pr}_2:\Sigma_{02}\to \Sigma_{12}$ is a submersion. Hence, $\bs_{02},\bt_{02}$ are compositions of submersions. The same sequence of implications shows that these maps are also surjective.

We turn next to the characterization provided by item $iv)$ of Proposition \ref{pro : equivalent wdp}. From
\[
\mathcal{R}_{-\omega_{01}}(\bs_{01}^!(L_0)) = \bt_{01}^!(L_1), \quad \bs_{12}^!(L_1) = \mathcal{R}_{\omega_{12}}(\bt_{12}^!(L_2)),
\]
and from the equality $\bt_{01}\circ\pr_1=\bs_{12}\circ\pr_2$ on $\Sigma_{02}$, it follows that
\[
\mathcal{R}_{-\mathrm{pr}_1^*(\omega_{01})}(\bs_{02}^!(L_0)) = \mathcal{R}_{\mathrm{pr}_2^*(\omega_{12})}(\bt_{12}^!(L_2)) \quad \iff \quad \bs_{02}^!(L_0) = \mathcal{R}_{\omega_{02}}(\bt_{02}^!(L_2)).
\]
%
Thus the first condition of Proposition \ref{pro : equivalent wdp} $iv)$  holds. We concllude the proof by showing that also the second condition in Proposition \ref{pro : equivalent wdp} $iv)$ hold. Let $a_0 \in L_0$. Then there is $w_{01} \in W_{01}:=\ker (\bt_{01*})$ such that $w_{01}+\iota_{w_{01}}\omega_{01} \sim_{\bs_{01}} a_0$. Then $w_{02}:=(w_{01},0) \in T\Sigma_{01} \times_{TM_1} T\Sigma_{12}$ lies in $W_{02}:=\ker (\bt_{02*})$, and $w_{02} + \iota_{w_{02}}\omega_{02} \sim_{\bs_{02}} a_0$. Similarly, for all $a_1 \in -L_1$, there is $v_{02} \in V_{02}:=\ker (\bs_{02*})$ such that $v_{02} + \iota_{v_{02}}\omega_{02} \sim_{\bt_{02}} a_1$. Hence
\[
w_{02}+v_{02}+\iota_{w_{02}+v_{02}}\omega_{02} \sim_{(\bs_{02},\bt_{02})} (a_0,a_1).
\]This concludes the proof.
\end{proof}

\begin{remark}
Given dual pairs
\begin{align*}
(M_0,L_0) \stackrel{\bs_{01}}{\lmap} (\Sigma_{01},\mathrm{Gr}(\omega_{01})) \stackrel{\bt_{01}}{\rmap} (M_1,-L_1)\\
(M_1,L_1) \stackrel{\bs_{12}}{\lmap}  (\Sigma_{12},\mathrm{Gr}(\omega_{12})) \stackrel{\bt_{12}}{\rmap} (M_2,-L_2),
\end{align*}
their composition (in the sense of Proposition \ref{lem : composition}) is not a dual pair, unless $\dim M_1 =0$, simply because
\[
 \dim \Sigma_{02} = \dim M_0 + \dim M_1 + \dim M_2.
\]
and hence the dimension condition of Definition \ref{def : dual pair} is violated.

When the foliation $\ker\bs_{02*} \cap \ker \omega_{02} \cap \ker\bt_{02*} \subset T\Sigma_{02}$ is simple, the reduction procedure of Proposition \ref{prop: reduction of dual pairs} $ii)$ below, applied to their composition, yields again a dual pair. Note however that this requirement is not always met, as illustrated in Example \ref{ex : composition which is not reducible} below.
\end{remark}

\subsection*{Reduction to dual pairs}

In the following, we describe how weak dual pairs can be pulled back via surjective submersions, and how, under certain assumptions (which hold locally), weak dual pairs can be \emph{reduced} to dual pairs.

\begin{proposition}[Reduction to dual pairs]\label{prop: reduction of dual pairs}
Consider a weak dual pair diagram\[(M_0,L_0) \stackrel{\bs}{\lmap} (\Sigma,\mathrm{Gr}(\omega)) \stackrel{\bt}{\rmap} (M_1,-L_1).\]
\begin{enumerate}[i)]
\item If $\mathrm{r}:\widetilde{\Sigma} \to \Sigma$ is a surjective submersion, then
\[
(M_0,L_0) \stackrel{\bs \circ \mathrm{r}}{\lmap} (\widetilde{\Sigma},\mathrm{Gr}(\mathrm{r}^*(\omega))) \stackrel{\bt \circ \mathrm{r}}{\rmap} (M_1,-L_1)
\]is again a weak dual pair.
\item
Assume that the foliation $V\cap K\cap W$ is simple, i.e.,\ that there exists a surjective submersion $\mathrm{r}:\Sigma\to \overline{\Sigma}$ whose fibres are the leaves of $V\cap K\cap W$. Then there is an induced commutative diagram of surjective, forward Dirac submersions:
\[
\xymatrixrowsep{0.4cm}
\xymatrixcolsep{1.2cm}
\xymatrix{
 & (\Sigma,\mathrm{Gr}(\omega)) \ar[dr]^{\bt}\ar[dl]_{\bs}\ar[d]^{\mathrm{r}} &  \\
 (M_0,L_0) & (\overline{\Sigma},\mathrm{Gr}(\overline{\omega})) \ar[r]^{\overline{\bt}}\ar[l]_{\overline{\bs}} &(M_1,-L_1)}\]
where $\overline{\omega}$ is a closed two-form on $\overline{\Sigma}$ such that $\omega=\mathrm{r}^*(\overline{\omega})$, and where the bottom line is a dual pair.
\end{enumerate}
\end{proposition}
\begin{proof}
We use the description of weak dual pairs provided by item $iv)$ of Proposition \ref{pro : equivalent wdp}. The first condition of $iv)$ holds, since
\[
 (\bs \circ \mathrm{r})^!(L_0) = \mathrm{r}^!(\bs^!(L_0)) = \mathrm{r}^!(\mathcal{R}_{\omega}(\bt^!(L_1))) = \mathcal{R}_{\mathrm{r}^*(\omega)}((\bt \circ \mathrm{r})^!(L_1)),
\]whereas the second holds because, for any surjective submersion $\mathrm{r}:\widetilde{\Sigma} \to \Sigma$, we have that $\mathrm{r}:(\widetilde{\Sigma},\mathrm{Gr}(\mathrm{r}^*\omega)) \to (\Sigma,\mathrm{Gr}(\omega))$ is forward Dirac. This proves $i)$.

By Lemma \ref{lem : VWK}, $V\cap K\cap W$ is a smooth, involutive distribution. For $ii)$, assume the existence of a surjective submersion $\mathrm{r}:\Sigma\to \overline{\Sigma}$ whose fibres are the leaves of $V\cap K\cap W$. By applying Proposition \ref{prop: Libermann simple} to the three Dirac structures corresponding to the foliations $V$ and $W$ and to the closed two-form $\omega$, we deduce that there are foliations $\overline{V}$ and $\overline{W}$, and a closed two-form $\overline{\omega}$ on $\overline{\Sigma}$, such that
\[V=\mathrm{r}^{-1}_*\overline{V}, \quad W=\mathrm{r}_*^{-1}\overline{W}, \quad \omega=\mathrm{r}^*\overline{\omega}.\]Since $\ker\mathrm{r}_*\subset V\cap W$, and the fibres of $\mathrm{r}$ are connected, it follows that $\bs$ and $\bt$ are constant along the fibres of $\mathrm{r}$; hence they factor as in the diagram, $\bs=\overline{\bs}\circ\mathrm{r}$, $\bt=\overline{\bt}\circ\mathrm{r}$, and these equalities imply that the maps $\overline{\bs}$ and $\overline{\bt}$ are surjective submersions. Note next that both these maps push forward the Dirac structure $\overline{L}:=\mathrm{Gr}(\overline{\omega})$:
\[\overline{\bs}_{!}(\overline{L})=\overline{\bs}_{!}\mathrm{r}_!(L)=\bs_{!}(L)=L_0, \quad \overline{\bt}_{!}(\overline{L})=\overline{\bt}_{!}\mathrm{r}_!(L)=\bt_{!}(L)=-L_1,\]
where $L:=\mathrm{Gr}(\omega)$. Moreover, the equalities:
\[\overline{K}=\ker \overline{\omega}=\mathrm{r}_*K, \ \ \overline{V}=\mathrm{r}_*V, \ \ \overline{W}=\mathrm{r}_{*}W,\]
imply that
\[\overline{\omega}(\overline{V},\overline{W})=\overline{\omega}(\mathrm{r}_*V,\mathrm{r}_*W)
=\omega(V,W)=0,\]
\[\overline{V}\cap \overline{K}\cap \overline{W}=0.\]
Since
\[\dim(\overline{\Sigma})=\dim (\Sigma)-\mathrm{rank}(V\cap K\cap W)=\dim (M_0)+\dim (M_1),\]
all the conditions from Definition \ref{def : dual pair} are met, hence the diagram is a dual pair.
\end{proof}

\begin{example}\label{ex : composition which is not reducible}
 Consider the dual pairs
 \begin{align*}
  (\mathbb{S}^1,T^*\mathbb{S}^1) \stackrel{\mathrm{pr}_2}{\lmap} (\mathbb{S}^1 \times \mathbb{S}^1,\mathrm{Gr}(\omega_{01})) \stackrel{\mathrm{pr}_1}{\rmap} (\mathbb{S}^1,T^*\mathbb{S}^1), \\
  (\mathbb{S}^1,T^*\mathbb{S}^1) \stackrel{\mathrm{pr}_2}{\lmap} (\mathbb{S}^1 \times \mathbb{S}^1,\mathrm{Gr}(\omega_{12})) \stackrel{\mathrm{pr}_1}{\rmap} (\mathbb{S}^1,T^*\mathbb{S}^1) \\
  \omega_{01} = \mathrm{pr}_1^*(\dd \theta) \wedge \mathrm{pr}_2^*(\dd \theta), \quad \omega_{12} = \lambda\mathrm{pr}_1^*(\dd \theta) \wedge \mathrm{pr}_2^*(\dd \theta),
 \end{align*}where $\lambda \in \mathbb{R} \diagdown \mathbb{Q}$ and where $\dd \theta \in \Omega^1(\mathbb{S}^1)$ denotes a volume form. Their composition as weak dual pairs is isomorphic with the diagram
 \begin{align*}
  (\mathbb{S}^1,T^*\mathbb{S}^1) \stackrel{\mathrm{pr}_2}{\lmap} (\mathbb{S}^1 \times \mathbb{S}^1 \times \mathbb{S}^1,\mathrm{Gr}(\omega_{02})) \stackrel{\mathrm{pr}_1}{\rmap} (\mathbb{S}^1,T^*\mathbb{S}^1).
 \end{align*}The leaves of the foliation $\ker\bs_{02*} \cap \ker \omega_{02} \cap \ker\bt_{02*} = \la \lambda \partial_{\theta_1}+\partial_{\theta_3}\ra$ on $\mathbb{S}^1 \times \mathbb{S}^1 \times \mathbb{S}^1$ are not closed; in particular, the foliation is not simple, and therefore the weak dual pair above does not reduce to a dual pair. (See also \cite[p. 39]{BursztynWeinstein})
\end{example}

Let us give a simple characterization of the situation when one of the legs of a weak dual pair is a Poisson structure.

\begin{lemma}\label{lem : poisson}
Let $(M_0,L_0) \stackrel{\bs}{\lmap} (\Sigma,\mathrm{Gr}(\omega)) \stackrel{\bt}{\rmap} (M_1,-L_1)$ be a weak dual pair.
 \begin{enumerate}[i)]
\item Then $L_1$ is a Poisson structure if and only if $W=V^{\omega}$.
\item If $L_1$ is a Poisson structure, then $\bs$ is a presymplectic realization iff the weak dual pair is a dual pair.
\end{enumerate}
\end{lemma}
\begin{proof}
$L_1$ is a Poisson structure exactly when $L_1 \cap TM_1 = 0$; but note that
\begin{align*}
L_1 \cap TM_1  &= \bt_!\bt^!(L_1) \cap TM_1 = \bt_!(\mathcal{R}_{-\omega}(V)+W) \cap TM_1 = \\
&= \bt_!(\mathcal{R}_{-\omega}(V)) \cap TM_1= \bt_*(V \cap K).
\end{align*}
Hence, $L_1$ is Poisson iff $V \cap K \subset W$. By Lemma \ref{lem : equivalent weak}, we have $V^{\omega}=W+V\cap K$. Therefore, $V \cap K \subset W$ iff $V^{\omega} = W$. This proves $i)$.

Let us prove now $ii)$. By $i)$ we have that $V\cap K\subset W$. Therefore $V \cap K \cap W=0$ iff $V\cap K=0$. The first condition is equivalent to the diagram being a dual pair, while the second is equivalent to $\bs$ being a presymplectic realization.
\end{proof}

Proposition \ref{prop: reduction of dual pairs} and Lemma \ref{lem : poisson} give a procedure for reducing to presymplectic realizations:

\begin{corollary}[Reduction to presymplectic realizations]\label{coro : presymplectic reduction}
Let \[\bs:(\Sigma,\mathrm{Gr}(\omega))\rmap (M_0,L_0)\] be a surjective, forward Dirac submersion. Assume that $V^{\omega}$ is a smooth distribution. Then $V^{\omega}$ and $V\cap K$ are both smooth, involutive distributions, and:
\begin{enumerate}[i)]
\item if the foliation $V\cap K$ is simple, i.e.\ if its leaves are the fibres of a surjective submersion $\mathrm{r}:\Sigma\to \overline{\Sigma}$, then $\omega=\mathrm{r}^*(\overline{\omega})$, where $\overline{\omega}$ is a closed 2-form on $\overline{\Sigma}$, and $\bs$ factors as $\bs=\overline{\bs}\circ \mathrm{r}$, where $\overline{\bs}$ is a presymplectic realization
    \[
\xymatrixrowsep{0.4cm}
\xymatrixcolsep{1.2cm}
\xymatrix{
 & (\Sigma,\mathrm{Gr}(\omega)) \ar[dl]_{\bs}\ar[d]^{\mathrm{r}} &  \\
 (M_0,L_0) & (\overline{\Sigma},\mathrm{Gr}(\overline{\omega})) \ar[l]_{\overline{\bs}} &\textrm{\phantom{$(M_1,\mathrm{Gr}(-\pi_1))$}}}
 \]
 \item if the foliation $V^{\omega}$ is simple, i.e.\ if its leaves are the fibres of a surjective submersion $\bt:\Sigma\to M_1$, then there exists a Poisson structure $\pi_1$ on $M_1$ which fits into the weak dual pair
\[(M_0,L_0) \stackrel{\bs}{\lmap} (\Sigma,\mathrm{Gr}(\omega)) \stackrel{\bt}{\rmap} (M_1,\mathrm{Gr}(-\pi_1))\]
 \item if both $i)$ and $ii)$ are assumed, then $\bt$ factors as $\bt=\overline{\bt}\circ\mathrm{r}$, and the weak dual pair from $ii)$ reduces to a dual pair:
    \[
\xymatrixrowsep{0.4cm}
\xymatrixcolsep{1.2cm}
\xymatrix{
 & (\Sigma,\mathrm{Gr}(\omega)) \ar[dr]^{\bt}\ar[dl]_{\bs}\ar[d]^{\mathrm{r}} &  \\
 (M_0,L_0) & (\overline{\Sigma},\mathrm{Gr}(\overline{\omega})) \ar[r]^{\overline{\bt}}\ar[l]_{\overline{\bs}} &(M_1,\mathrm{Gr}(-\pi_1))}\]
\end{enumerate}
\end{corollary}

\begin{example}
Consider the forward Dirac map $\bs:(\mathbb{R}^3,\mathrm{Gr}(\omega))\to (\mathbb{R}^2,L_0)$ from Example \ref{example : pre-dual pair but not weak dual pair}. Then $\bs$ is a presymplectic realization, i.e.\ $V\cap K=0$, and $V^{\omega}$ is the simple foliation corresponding to the fibres of $\overline{\bt}:\mathbb{R}^3\to \mathbb{R}$, $\overline{\bt}(x,y,z):=y$. By Corollary \ref{coro : presymplectic reduction}, we obtain the dual pair (compare with the pre-dual pair of Example \ref{example : pre-dual pair but not weak dual pair}):
\[(\mathbb{R}^2,L_1) \stackrel{\bs}{\longleftarrow} (\mathbb{R}^3,\mathrm{Gr}(\omega)) \stackrel{\overline{\bt}}{\longrightarrow} (\mathbb{R},T^*\mathbb{R}).\]
\end{example}

\subsection*{Pullback along transverse maps}

In the Poisson setting, the procedure of pullback of a dual pair to small transversals to the symplectic leaves goes back to \cite[Theorem 8.1]{Weinstein83}, where it is shown that the induced transverse Poisson structures are anti-isomorphic. Here we discuss the procedure in the general Dirac setting.

We will use the following result.

\begin{lemma}\label{lem : transv}
Let $f:(M,L_M)\to (N,L_N)$ be a forward Dirac map, and let $g:X\to N$ be a smooth map that is transverse to $L_N$ (see \S 12). Then the following hold:
\begin{enumerate}[i)]
\item $f$ and $g$ are transverse maps, so that $M\times_{N}X$, i.e.\ the pullback of $f$ and $g$, is a smooth manifold;
\item $\mathrm{pr}_1:M\times_{N}X\to M$ is transverse to $L_M$, so that $\mathrm{pr}_1^{!}(L_M)$ is a Dirac structure on $M\times_{N}X$;
\item $\mathrm{pr}_2:(M\times_{N}X,\mathrm{pr}_1^{!}(L_M))\to (X,g^{!}(L_N))$ is forward Dirac. Moreover, if $f$ is a strong map, then also $\mathrm{pr}_2$ is a strong map.
\end{enumerate}
\end{lemma}
\begin{proof}
Throughout, $(m,x)$ will denote a point in $M\times_NX$, so that $n:=f(m)=g(x)$. Using that $g$ is transverse to $L_N$ and that $f$ is forward Dirac, for , we obtain:
\[T_nN=g_*(T_xX)+\mathrm{pr}_T(L_{N,n})=g_*(T_xX)+\mathrm{pr}_T(f_{!}L_{M,m})\subset g_*(T_xX)+f_*(T_mM),\]
which proves $i)$.

For $ii)$, let $v\in T_mM$. By item $i)$ above, we can decompose $f_*(v)=g_*(u)+f_*(w)$, with $u\in T_xX$ and $w+\alpha\in L_{M,m}$ for some $\alpha\in T^*_mM$. Then $(v-w,u)\in T_{(m,x)}(M\times_NX)$ and $v=\mathrm{pr}_1(v-w,u)+\mathrm{pr}_T(w+\alpha)$. This proves $ii)$.

The first part of $iii)$ is proven in \cite[Lemma 3a)]{PT1.5}. Finally, assume that $f$ is a strong forward Dirac map. Consider $(v,0)\in T_{(m,x)}(M\times_NX)\cap \mathrm{pr}_1^!(L_M)$. Then $f_*(v)=g_*(0)=0$. Since $\mathrm{pr}_1$ is backward Dirac, there is $\alpha\in \ker \mathrm{pr}_1^*$ such that $v+\alpha\in L_{M,m}$. By transversality of $f$ and $g$, we have the short exact sequence:
\[0\rmap T^*_nN\stackrel{f^*-g^*}{\rmap}T^*_mM\oplus T^*_xX\stackrel{(\mathrm{pr}_1^*,\mathrm{pr}_2^*)}{\rmap}T^*_{(m,x)}(M\times_NX)\rmap 0.\]
Hence, there is $\beta\in T^*_nN$ such that $f^*(\beta)=\alpha$ and $g^*(\beta)=0$. Since $f$ is forward Dirac, we have that $f_*(v)+\beta=\beta\in L_{N,n} \cap \ker g^*$. Since $g$ is transverse to $L_N$, this yields $\beta=0$, and so $\alpha=0$. So $v\in L_{M,m}\cap \ker f_*$, and since $f$ is a strong map, $v=0$. We have shown that $\ker \mathrm{pr}_{2*}\cap \mathrm{pr}_1^!(L_M)=0$, i.e.\ $\mathrm{pr}_2$ is a strong map.
\end{proof}

Next, we present the general procedure of pulling back along transverse maps.

\begin{proposition}[Transverse pullback]\label{prop: transverse_pullback}
Consider a weak dual pair
\[(M_0,L_0) \stackrel{\bs}{\lmap} (\Sigma,\mathrm{Gr}(\omega)) \stackrel{\bt}{\rmap} (M_1,-L_1).\]
For $j=0,1$, consider a smooth map $i_j:X_j\to M_j$ that is transverse to $L_j$. Denote
\[\Sigma_X:=X_0\times_{M_0}\Sigma\times_{M_1}X_1,\ \ \ \omega_X:=\mathrm{pr}_2^*\omega\]
\[\bs_X=\mathrm{pr}_1:\Sigma_X\to X_0,\ \ \ \bt_X=\mathrm{pr}_3:\Sigma_X\to X_1.\]
Then the diagram
\[(X_0,i_0^!(L_0)) \stackrel{\bs_X}{\lmap} (\Sigma_X,\mathrm{Gr}(\omega_X)) \stackrel{\bt_X}{\rmap} (X_1,-i_1^!(L_1))\]
satisfies all axioms of a weak dual pair, except maybe for the surjectivity of the maps $\bs_X$ and $\bt_X$. The same holds for dual pairs instead of weak dual pairs.
\end{proposition}
\begin{proof}
We apply Lemma \ref{lem : transv} to the pair of maps: \[(\bs,\bt):(\Sigma,\mathrm{Gr}(\omega))\rmap (M_0,L_0)\times (M_1,-L_1), \]
\[i_0\times i_1:X_0\times X_1\rmap M_0\times M_1.\]
We conclude that $\Sigma_X$ is a smooth manifold, and that
\[(\bs_X,\bt_X):(\Sigma_X,\mathrm{Gr}(\omega_X))\rmap (X_0,i_0^!(L_0))\times (X_1,-i_1^!(L_1))\]
is a forward Dirac map.

Next, we show that $\bs_X$ and $\bt_X$ are submersions. We use the notation $V:=\ker\bs_*$, $W:=\ker \bt_*$ and $K:=\ker\omega$. Let $(x_0,z,x_1)\in \Sigma_{X}$, and let $v_0\in T_{x_0}X_0$. Since $\bs$ is a submersion, there is $w_0\in T_{z}\Sigma$ such that $\bs_*(w_0)=i_{0*}(v_0)$. Since $i_1$ and $L_1$ are transverse, we can decompose $\bt_*(w_0)=i_{1*}(v_1)+u$, where $u+\alpha\in L_{1,m_1}$ and $m_1=i_1(x_1)$. Since $\bt$ is forward Dirac, there is $w_1\in T_{z}\Sigma$ such that $\bt_*(w_1)=u$ and $\iota_{w_1}\omega=-\bt^*(\alpha)$. This implies that $w_1\in W^{\omega}$. By Lemma \ref{lem : equivalent weak} (c), $W^{\omega}=V+W\cap K$, so we can decompose $w_1=w_2+w_3$ with $w_2\in V$ and $w_3\in W$. Hence $\bt_*(w_2)=\bt_*(w_1)=u$ and $\bs_*(w_2)=0$. So the element $w=w_0-w_2$ satisfies $\bs_*(w)=i_{0*}(v_0)$ and $\bt_*(w)=i_{1*}(v_1)$, thus $(v_0,w,v_1)\in T_{(x_0,z,x_1)}\Sigma_X$. This shows that $\bs_X$ is a submersion, and similarly, one shows that also $\bt_X$ is a submersion.

Since the fibres of $\bs$ and $\bt$ are $\omega$-orthogonal, by construction it follows that also the fibres of $\bs_X$ and $\bt_X$ are $\omega_X$-orthogonal; hence it follows from Proposition \ref{pro : equivalent wdp} $iii)$ that the transverse pullback
\[(X_0,i_0^!(L_0)) \stackrel{\bs_X}{\lmap} (\Sigma_X,\mathrm{Gr}(\omega_X)) \stackrel{\bt_X}{\rmap} (X_1,-i_1^!(L_1))\]is a weak dual pair.

Assume now that the weak dual pair in the statement is a dual pair. This is equivalent by Proposition \ref{pro : equivalent dual pair} $iii)$ to $(\bs,\bt)$ being a presymplectic realization and $\dim \Sigma = \dim M_0+\dim M_1$. In that case, it follows from Lemma \ref{lem : transv} that $(\bs_X,\bt_X)$ is also a presymplectic realization, and transversality yields the dimension condition
\begin{align*}
\mathrm{dim}(\Sigma_X)&=\mathrm{dim}(\Sigma)+\mathrm{dim}(X_0)+\mathrm{dim}(X_1)-\mathrm{dim}(M_0)-\mathrm{dim}(M_1)\\
&=\mathrm{dim}(X_0)+\mathrm{dim}(X_1).
\end{align*}Invoking Proposition \ref{pro : equivalent dual pair} $iii)$ once again, we conclude that the transverse pullback of a dual pair is again a dual pair.
\end{proof}

\section{On the existence of self-dual pairs}\label{sec : On the existence of self-dual pairs}

An important ingredient of our Main Theorem is the following notion:

\begin{definition}\label{def: spray}
Let $L \subset \TT M$ be a Dirac structure on $M$, and let $\bs:L \to M$ denote the bundle projection. A \textbf{spray} for $L$ is a vector field $\mathcal{V} \in \mathfrak{X}(L)$, satisfying:
\begin{itemize}
 \item[(Spr1)] $\bs_{\ast}\mathcal{V}_a = \mathrm{pr}_{T}(a)$, for all $a \in L$;
 \item[(Spr2)] $m_t^{\ast}\mathcal{V}=t\mathcal{V}$, where $m_t:L\to L$ denotes multiplication by $t\neq 0$.
\end{itemize}
\end{definition}
For example, given a linear connection on $L$ with horizontal lift $h$, the vector field $\mathcal{V}_a:=h_a(\mathrm{pr}_{T}(a))$ is a spray for $L$.

Given a spray $\mathcal{V}$ for $L$, note that (Spr2) implies that $\mathcal{V}$ vanishes along $M$, identified with the zero-section of $L$. Therefore, there exists a small enough neighborhood $\mathcal{U}\subset L$ of $M$ on which the flow $\varphi_{\epsilon}:\mathcal{U} \to L$ of $\mathcal{V}$ is defined for $0\leq \epsilon\leq 1$, and consider
 \[\omega:=\int_0^1\varphi_{\epsilon}^{\ast}\omega_L \mathrm{d}\epsilon \ \in \ \Omega^2(\mathcal{U}), \quad \bt:=\bs \circ \varphi_1:\mathcal{U} \to M,
 \]
where $\omega_L \in \Omega^2(L)$ denotes the pullback of the canonical two-form on $T^*M$ under $\mathrm{pr}_{T^*}:L \to T^*M$. With this notation, we state:

\begin{maintheorem}\label{thm : Dirac dual pairs}
There is an open set $\Sigma\subset \mathcal{U}$ containing $M$, such that the diagram
\[(M,L) \stackrel{\bs}{\lmap} (\Sigma,\mathrm{Gr}(\omega)) \stackrel{\bt}{\rmap} (M,-L)\]
forms a dual pair.
\end{maintheorem}

The proof will make use of the following

\begin{lemma}\label{lem : flow linearized}
\begin{enumerate}[i)]
\item The flow $\varphi_{\epsilon}$ is the identity along $M$, and in the canonical decomposition $TL|_M=TM \oplus L$, its differential reads
\[
\varphi_{\epsilon*}:TL|_M \diffto TL|_M, \quad \varphi_{\epsilon*}(u,a) = (u+\epsilon \mathrm{pr}_T(a),a).
\]
\item We have that $(\ker \bs_* \cap \ker \omega \cap \ker \bt_*)|_M = 0$.
\end{enumerate}
\end{lemma}
\begin{proof}
$i)$ Since $\mathcal{V}$ vanishes along $M$, we have $\varphi_{\epsilon}|_M=\mathrm{id}_M$. Let $\mathcal{V}_T \in \mathfrak{X}(TL)$ be the vector field generating the flow $\varphi_{\epsilon*}:TL \diffto TL$ (also called the \emph{tangent lift} of $\mathcal{V}$). Due to (Spr2), $\mathcal{V}_T$ is tangent to $T_{x}L\subset TL$ for each $x \in M$, and $\mathcal{V}_T|_{T_{x}L} \in \mathfrak{X}(T_{x}L)$ is a linear vector field. By (Spr1), this vector field corresponds to the endomorphism
\[T_xM \oplus L_x \to T_xM \oplus L_x, \quad (u,a) \mapsto (\mathrm{pr}_T(a),0);\]
    the exponential of this map, i.e.\ $\varphi_{\epsilon*}$ on $T_xM\oplus L_x$, is given by the formula in the statement. For a different proof, see \cite[Lemma 3.21]{BCL-D}.

$ii)$ Along the zero section $M \subset L$, the two-form $\omega_L \in \Omega^2(L)$ reads:
\[\omega_L\left((u,v+\eta),(u',v'+\eta')\right)=\eta'(u)-\eta(u').\]
On the other hand, by item $i)$, we have
\begin{equation*}
  \varphi_{\epsilon\ast}(u,v+\eta)= (u+\epsilon v,v+\eta), \ \ (u,v+\eta)\in TL|_M = TM\oplus L;
\end{equation*}
hence
\begin{align}\label{eq : form}
\omega\left((u,v+\eta),(u',v'+\eta')\right)=\eta'(u+1/2v)-\eta(u'+1/2v').
\end{align}
Let $(u,v+\eta)\in \ker \bs_* \cap \ker \omega \cap \ker \bt_*|_M$. Since $\bs_{*}(u,v+\eta)=u$ and
\begin{align*}
\bt_{*}(u,v+\eta)=\bs_{*}\varphi_{1*}(u,v+\eta)=\bs_{*}(u+v,v+\eta)=u+v,
\end{align*}
it follows that $u=0$ and $v=0$. By formula (\ref{eq : form}),
\[0=\omega((u,v+\eta),(u',0))=-\eta(u')\]
for all $u'$; thus, also $\eta=0$.\qedhere
\end{proof}

\begin{proof}[Proof of Main Theorem]
Let $\mathcal{U}\subset L$ be a neighborhood of $M$ on which $\varphi_{\epsilon}$ is defined for $0\leq \epsilon\leq 1$, and so also $\bt=\bs \circ \varphi_1$ and $\omega=\int_0^1\varphi_{\epsilon}^*\omega_{L} \dd\epsilon$ are defined on $\mathcal{U}$. Write $V,W,K \subset T\mathcal{U}$ for $\ker \bs_*,\ker \bt_*$ and $\ker \omega$, respectively. By Lemma \ref{lem : flow linearized}$ii)$, $(V\cap K \cap W)|_M=0$. Sine this condition is open, we may assume that
\begin{equation}\label{inca una}
(V\cap K \cap W)|_{\Sigma}=0,
\end{equation}
where $\Sigma\subset \mathcal{U}$ is an open neighborhood of $M$.

Consider the pullback to $L$ of the tautological one-form $\lambda_{\mathrm{can}}$ on $T^*M$:
\[\lambda_L:=(\mathrm{pr}_{T})^*\lambda_{\mathrm{can}}\in \Omega^1(L).\]
Note that condition (Spr1) in the definition of a spray $\mathcal{V}$ implies that $\mathcal{V}+\lambda_L$ is a section of the Dirac structure $\bs^!(L)$. The local flow $\Phi_{\epsilon}$ of the section $\mathcal{V}+\lambda_L$ on $\bs^!(L)$ covers the local flow $\varphi_{\epsilon}:L \to L$ and, on its domain of definition, it is given by (e.g.\ \cite[Proposition 2.3]{Marco}):
\[\Phi_{\epsilon}=\varphi_{\epsilon*}\circ \mathcal{R}_{B_{\epsilon}}:\bs^!(L)\rmap \bs^!(L),\]
where $B_{\epsilon}=\int_0^{\epsilon} \varphi_{s}^{\ast}\mathrm{d}\lambda_L \mathrm{d}s=-\int_0^{\epsilon} \varphi_{s}^{\ast}\omega_L \mathrm{d}s$. Since $B_1=-\omega$, we have that: \[\bs^!(L)|_{\varphi_{1}(\Sigma)}=\Phi_{1}\left(\bs^!(L)|_{\Sigma}\right)=\varphi_{1!}(\mathcal{R}_{-\omega}(\bs^!(L)|_{\Sigma})),\]
and therefore
\begin{equation}\label{eq : gauge condition of main theorem}
\mathcal{R}_{-\omega}(\bs^!(L)|_{\Sigma})=\varphi_1^!(\bs^!(L)|_{\varphi_1(\Sigma)})=\bt^!(L)|_{\Sigma}.
\end{equation}
Since $\dim \Sigma = 2 \dim M$ and equations (\ref{inca una}) and (\ref{eq : gauge condition of main theorem}) hold, Proposition \ref{pro : equivalent dual pair} $ii)\Rightarrow i)$ yields the conclusion.
\end{proof}

\section{Application: Normal forms around Dirac transversals}\label{sec : dirac transversals}

In this section, we use the Main Theorem to prove the normal form theorem around Dirac transversals (from \cite{BLM}) The same strategy is already present in \cite{PT1} in the context of Poisson geometry.

\begin{definition}\label{def : dirac transversal}
An embedded submanifold $X$ of a Dirac manifold $(M,L)$ is called a \emph{Dirac transversal} if the inclusion $i:X\to M$ is transverse to $L$.
\end{definition}

The normal form theorem around Dirac transversals states that, up to diffeomorphisms and exact gauge transformations, the induced Dirac structure on the transversal determines the Dirac structure around the transversal. The Poisson geometric version of this result appeared in \cite{PT1}; a local version for Dirac structures occurs in \cite{Blo}; a version for generalized complex structures was described in \cite{BCL-D}. Moreover, the same statement as the one below, but with a different proof, appeared in \cite{BLM}:

\begin{theorem*}[Normal form around Dirac transversals]\label{thm : Normal form for Dirac transversals}
Let $i:X \to (M,L)$ be a Dirac transversal, and let $p:NX\to X$ be the normal bundle of $X$. Then there exist an embedding $\varphi : U \to M$ of an open neighborhood $U \subset NX$ of $X$ extending $i$, and an exact two-form $\alpha\in \Omega^2(U)$, such that
 \[\varphi^!(L) = \mathcal{R}_{\alpha}(p^!i^!(L)).\]
\end{theorem*}
\begin{proof}
Consider the dual pair given by the Main Theorem:
\[ (M,L) \stackrel{\mathbf{s}}{\longleftarrow} (\Sigma,\mathrm{Gr}(\omega)) \stackrel{\mathbf{t}}{\rmap} (M,-L).\]
We apply the pullback procedure of Proposition \ref{prop: transverse_pullback}, and obtain the diagram
\begin{equation}\label{eq : sigma 0 dual pair}
 (X,i^!(L)) \stackrel{\mathbf{s}_0}{\longleftarrow} (\Sigma_0,\mathrm{Gr}(\omega_0)) \stackrel{\mathbf{t}_0}{\rmap} (M,-L),
\end{equation}
where
\[\Sigma_{0}:=\bs^{-1}(X), \ \ \omega_0:=\omega|_{\Sigma_0},\ \ \bs_0:=\bs|_{\Sigma_0}, \ \ \bt_0:=\bt|_{\Sigma_0}.\]
Note that $\bs_0$ is surjective. By construction, the image of $\bt_0$ contains $X$ and it is open, because $\bt_0$ is a submersion. Therefore, by replacing $M$ by $\bt_0(\Sigma_0)$, we may assume that also $\bt_0$ is onto, and so, by Proposition \ref{prop: transverse_pullback}, we may assume that the diagram (\ref{eq : sigma 0 dual pair}) is a dual pair.

The fact that $X$ is a Dirac transversal implies that
\[
\widetilde{i}: i^!(L) \rmap L|_X, \quad \widetilde{i}(u+i^*\xi):=i_{*}(u)+\xi
\]
is a well-defined bundle map which fits into the exact sequence:
\[0 \rmap i^!(L) \stackrel{\widetilde{i}}{\rmap} L|_X \rmap NX \rmap 0.\]
Let $\rho:NX \to i^*(L)$ be a splitting of this sequence. By Lemma \ref{lem : flow linearized} $i)$, we have that
\[
\bt_{0_*} : T\Sigma_0|_X = TX \oplus L \rmap TM, \quad \bt_{0_*}(u,a) = u+\mathrm{pr}_T(a),
\]
and therefore there exists an open set $X \subset U \subset NX$ such that
\[
\varphi := \bt_0 \circ \rho :U \hmap M
\]
is an open embedding extending the inclusion $i$. Since $\bs_0 \circ \rho = p$ and $\varphi = \bt_0 \circ \rho$,
\begin{align*}
 p^!(i^!(L)) = \rho^!(\bs_0^!(i^!(L))) = \rho^!\left(\mathcal{R}_{\omega_0}(\bt_0^!(L))\right) = \mathcal{R}_{\rho^*\omega_0}(\varphi^!(L)),
\end{align*}
where we have used the relation $\bs_0^!(i^!(L)) = \mathcal{R}_{\omega_0}(\bt_0^!(L))$ from Proposition \ref{pro : equivalent dual pair} $ii)$. This implies the conclusion:
\[ \varphi^!(L) = \mathcal{R}_{\alpha}(p^!(i^!(L))), \ \ \textrm{where} \ \ \alpha:=-\rho^*\omega_0 \in \Omega^2(U).\qedhere\]
\end{proof}

\section{Further remarks}\label{sec : further}

\subsection*{The origin of the formula}

The formula for the two-form $\omega$ from the Main Theorem originates in the path-space approach to integrability of Lie algebroids, developed in \cite{CrFer03}. In a nutshell, given a Lie algebroid $A$ over a manifold $M$, the space $P(A)$ of $A$-paths carries a canonical \emph{homotopy foliation}, of finite codimension. The leaf space of this foliation has a canonical structure of topological groupoid $G(A) \rightrightarrows M$ (the so-called \emph{Weinstein groupoid} of $A$), which is smooth exactly when $A$ is integrable by a Lie groupoid. If the Lie algebroid is a Dirac structure $A=L$ on $M$, and moreover, if it is integrable, then $G(L)$ carries a canonical, multiplicative, closed two-form $\omega_{G(L)}$, for which the source and target maps $\bs,\bt:(G(L),\mathrm{Gr}(\omega_{G(L)}))\to (M,L)$ give a dual pair \cite{BCWZ}. However, even if $L$ is not integrable, the Banach manifold $P(L)$ of $L$-paths carries a canonical two-form $\omega_{P(L)}$ which is basic for the homotopy foliation, and, in the integrable case, it is the pullback of $\omega_{G(L)}$. Using the language developed in Section \ref{sec : properties}, in the integrable case the passage from $P(L)$ to $G(L)$ can be viewed as a reduction of an infinite dimensional weak dual pair to a finite dimensional dual pair:
 \[
\xymatrixrowsep{0.4cm}
\xymatrixcolsep{1.2cm}
\xymatrix{
 & (P(L),\mathrm{Gr}(\omega_{P(L)})) \ar[dr]^{\bt}\ar[dl]_{\bs}\ar[d]^{\pi} &  \\
 (M,L) & (G(L),\mathrm{Gr}(\omega_{G(L)})) \ar[r]^{\bt}\ar[l]_{\bs} &(M,-L)}\]

Now, a spray $\mathcal{V}$ on $L$ induces an exponential map $\exp_{\mathcal{V}}:\Sigma\to P(L)$ on a neighborhood $\Sigma\subset L$ of the zero-section, which is transverse to the homotopy foliation, and of complementary dimension. The two-form $\omega$ from our main result is the precisely the pullback of $\omega_{P(L)}$ via $\exp_{\mathcal{V}}$
(see the explicit formula of $\omega_{P(L)}$ from \cite[Section 5]{BCWZ}). In fact, along these lines, one can use results of \emph{loc.cit.} to give an alternative proof of our Main Theorem; however, such a proof is bound to be less elementary.

\subsection*{Local presymplectic groupoids}

Closely related to the dual pair constructed in our Main Theorem is the notion of \emph{local presymplectic groupoid}. First, by \cite[Corollary 5.1]{CrFer03}, any Lie algebroid is integrable by a local Lie groupoid, and second, by extending the results of \cite{BCWZ} to the setting of such local objects, one can prove the existence of a local presymplectic groupoid integrating a given Dirac manifold; and in particular, the existence of a self-dual pair.

However, the converse construction seems of more significance: one can endow the self-dual pair constructed in the Main Theorem with the structure of a local presymplectic groupoid, yielding a rather explicit construction of a local integration of a Dirac manifold; a simple version of this construction was shown to us by Eckhard Meinrenken \cite{Eckhardnotes}, which uses the comorphism $\Psi$ from Remark \ref{rem : bracket relations} to construct an action of $L\times-L$ of $\Sigma$, which corresponds to the left and right invariant vector fields on the local Lie groupoid $\Sigma\rightrightarrows M$. This is very much in the spirit of \cite{CDW}, where a local symplectic groupoid integrating a Poisson manifold is constructed by using a symplectic realization. For a general construction of local Lie groupoids using Lie algebroid sprays, see \cite{CMS}.

\subsection*{The twisted case}

Courant algebroids may be \emph{twisted} by closed three-forms. Specifically, let $\phi \in \Omega^3(M)$ be closed, and consider the {\bf $\phi$-twisted} Dorfman bracket
\begin{equation*}
 [u+\xi,v+\eta]_{\phi}:=[u+\xi,v+\eta]+\iota_u\iota_v\phi, \quad u+\xi, \ v+\eta \in \Gamma(\TT M),
\end{equation*}
and where $[u+\xi,v+\eta]$ stands for the Dorfman bracket of \S 2. A Lagrangian subbundle $L \subset \TT M$ (in the sense of \S 6) is called a {\bf $\phi$-twisted Dirac structure} if its space of sections is involutive under the $\phi$-twisted Dorfman bracket.

Twisted Dirac structures go back, in one form or another, to \cite{WZW,Park,Dima,SeveraWeinstein}. A crucial example of such structures is given by \emph{Cartan-Dirac structures} associated to nondegenerate, invariant inner products on the Lie algebra $\mathfrak{g}$ of a Lie group $G$ \cite[Example 4.2]{SeveraWeinstein}, and whose presymplectic realizations correspond to the quasi-Hamiltonian $\mathfrak{g}$-spaces of \cite{AMM} (see \cite{BCWZ}).

Twisted Dirac structures are in particular Lie algebroids, whose global objects correspond (in the integrable case) to the \emph{twisted presymplectic groupoids} of \cite{BCWZ}, or \emph{quasi-symplectic groupoids} in the language of \cite{Xu2}.

Our Main Theorem can be adapted to the twisted setting in a straightforward way \cite{Eckhard} (see also the comments in \cite[Section 4]{CrMar11}). We briefly indicate here the necessary modifications. The notions of \emph{forward/backward Dirac maps}, and of \emph{sprays} make sense for general Lagrangian subbundles (as pointwise conditions), and retain their meaning from the untwisted case. For example, if $L \subset \TT M$ is a $\phi$-twisted Dirac structure, and $\bs:\Sigma \to M$ is a submersion, then $\bs^!(L):=\{u+\bs^*(\xi) \ | \ \bs_*(u)+\xi \in L\}$ is a $\bs^*(\phi)$-twisted Dirac structure on $\Sigma$. The notion of \emph{dual pair}, however, needs to be adapted: if $(M_i,L_i,\phi_i)$ are twisted Dirac structures, a {\bf twisted dual pair} is a diagram of surjective, forward Dirac submersions
\[
 (M_0,L_0) \stackrel{\bs}{\lmap} (\Sigma,\omega) \stackrel{\bt}{\rmap} (M_1,-L_1)
\]
as in Definition \ref{def : dual pair}, but where now $\omega \in \Omega^2(\Sigma)$ is a two-form satisfying
\begin{equation*}
 \dd\omega = \bs^*(\phi_0) - \bt^*(\phi_1).
\end{equation*}
Lemma \ref{lem : flow linearized} remains valid \emph{ipsis litteris} in the twisted case, and if $\mathcal{V} \in \mathfrak{X}(L)$ is a spray for $L$ (in the sense that it satisfies conditions (Spr1) and (Spr2) of Definition \ref{def: spray}), then $\mathcal{V}+\lambda_L$ is a section of the $\bs^*(\phi)$-twisted Dirac structure $\bs^!(L)$ (where $\lambda_L$ is, as before, the pullback of the tautological one-form on $M$ via the projection $\mathrm{pr}_{T^*}:L \to T^*M$). Then $D=[\mathcal{V}+\lambda_L,\cdot]_{\bs^*(\phi)}$ corresponds to the derivation
\[D=(\mathcal{V},b) \in \mathfrak{X}(L) \oplus \Omega^2(L), \quad b:=\dd\lambda_L+\iota_{\mathcal{V}}\bs^*(\phi), \quad \mathscr{L}_\mathcal{V}\bs^*(\phi) = \dd b,
\]that is, it integrates to a linear automorphism
\[\Phi_{\epsilon}:\TT \mathcal{U}\rmap \TT \mathcal{U}, \ \ \Phi_{\epsilon}(u+\eta)=\varphi_{\epsilon *}(u)+\varphi_{\epsilon}^{*-1}(\iota_uB_{\epsilon}+\eta),\]which is orthogonal for the inner product of \S 2, and preserves the $\bs^*(\phi)$-twisted bracket; here $\mathcal{U} \subset L$ denotes (as before) an open neighborhood of $M$ on which the flow of $\mathcal{V}$ is defined up to time one, and $B_{\epsilon}$ denotes the two-form $B_{\epsilon}:=\int_0^{\epsilon} \varphi_s^*(b) \dd s$, which satisfies $\dd B_{\epsilon}=\varphi_{\epsilon}^*\bs^*(\phi) - \bs^*(\phi)$.

The condition that $[\Gamma(\bs^!(L)),\Gamma(\bs^!(L))]_{\bs^*(\phi)} \subset \Gamma(\bs^!(L))$ implies that $D\Gamma(\bs^!(L)) \subset \Gamma(\bs^!(L))$, and so $\Phi_{\epsilon}$ restricts to a Lie algebroid automorphism of $\bs^!(L)$. As in the proof of the Main Theorem, this implies that
\[
 \bs^!(L) = \ker \bs_* \oplus \mathcal{R}_{\omega}(\ker \bt_*), \quad \bt^!(L) = \mathcal{R}_{-\omega}(\ker \bs_*) \oplus \ker \bt_*,
\]where $\omega:=-B_1$, $\bt:=\bs \circ \varphi_1$. The twisted version of Proposition \ref{pro : equivalent dual pair} allows us to deduce that, on an open neighborhood $\Sigma \subset \mathcal{U}$ of $M$, the map $(\bs,\bt):(\Sigma,\mathrm{Gr}(\omega)) \to (M,L) \times (M,-L)$ is a presymplectic realization; we obtain a twisted dual pair:
\[
   (M,L) \stackrel{\bs}{\lmap} (\Sigma,\mathrm{Gr}(\omega)) \stackrel{\bt}{\rmap} (M,-L), \quad \dd \omega = \bs^*(\phi)-\bt^*(\phi).
  \]

\subsection*{Morita equivalence}

Closely related to the notion of dual pairs above is that of \emph{Morita equivalences}. This notion of equivalence was first introduced by P. Xu \cite{Xu0,Xu1} in the context of Poisson manifolds as a classical analogue of the algebraic notion for $C^*$-algebras developed by M. Rieffel \cite{Rieffel}.

Two Poisson manifolds $(M_0,\pi_0)$, $(M_1,\pi_1)$ are called Morita equivalent
if there exists a dual pair
\[
 (M_0,\mathrm{Gr}(\pi_0)) \stackrel{\bs}{\lmap} (\Sigma,\mathrm{Gr}(\omega)) \stackrel{\bt}{\rmap} (M_1,-\mathrm{Gr}(\pi_1))
\]with simply-connected fibres, in which the induced comorphisms\[\bs^*(\mathrm{Gr}(\pi_0)) \to T\Sigma, \quad \bt^*(\mathrm{Gr}(\pi_1)) \to T\Sigma\]are complete \cite[Definition 2.1]{Xu1}. This ensures that the Poisson manifolds are integrable (see Corollary \ref{coro: int}), that the symplectic leaves of $(M_0,\pi_0)$ and $(M_1,\pi_1)$ are in bijection (with the same first homology), that they have anti-isomorphic transverse Poisson structures, isomorphic algebras of Casimirs and modules of first Poisson cohomology \cite{GinzburgLu,Weinstein83}. More importantly, it ensures that Morita equivalent Poisson manifolds have equivalent 'categories' of complete symplectic realizations \cite[Theorem 3.3]{Xu1}.

Morita equivalence of Poisson manifolds can also be approached through symplectic groupoids: two (integrable) Poisson manifolds are Morita equivalent if their source-simply connected symplectic groupoids $(G_0,\omega_0) \rightrightarrows (M_0,\pi_0)$, $(G_1,\omega_1) \rightrightarrows (M_1,\pi_1)$ admit a symplectic $(G_0,G_1)$-bibundle which is biprincipal \cite{BursztynWeinstein,Landsman,Xu0,Xu1}.	This line of thought was greatly extended in \cite{Xu2}, which introduces Morita equivalence for quasi-symplectic groupoids (or twisted presymplectic groupoids in the sense of \cite{BCWZ}), which plays an analogous role for twisted Dirac manifolds as Morita equivalence of symplectic groupoids in Poisson geometry. As in the Poisson case, this notion only makes sense in the integrable case, but some recent 'stacky' versions of Morita equivalence forgo the integrability hypothesis \cite{BursztynWeinstein_Picard,BursztynWeinstein,BursztynNosedaZhu} (cf. also the weaker notion of Morita equivalence discussed in \cite[Section 6]{Ginzburg}).


\begin{thebibliography}{9}
\bibitem{AMM}A.~Alekseev, A.~Malkin, E.~Meinrenken, \emph{Lie group valued moment maps}, J. Diff. Geom. 48 (1998), 445--495.
\bibitem{ABM}A.~Alekseev, H.~Bursztyn, E.~Meinrenken, \emph{Pure Spinors on Lie groups}, Ast\'{e}risque Volume 327 (2009), 131--199.
\bibitem{BCL-D} M.~Bailey, G.~Cavalcanti, J.~v.d.~Leer Duran, \emph{Blow-ups in generalized complex geometry}, preprint, \url{https://arxiv.org/abs/1602.02076}.
\bibitem{CMS} A.~Cabrera, I.~M\u{a}rcu\cb{t}, M.~A.~Salazar, \emph{A construction of local Lie groupoids using Lie algebroid sprays}, preprint, \url{https://arxiv.org/abs/1703.04411}.
\bibitem{Blo}C.~Blohmann, \emph{Removable presymplectic singularities and the local splitting of Dirac structures}, International Mathematics Research Notices (2016), \url{https://doi.org/10.1093/imrn/rnw238}
\bibitem{BraFern2} O.~Brahic, R.L.~Fernandes, \emph{Integrability and Reduction of Hamiltonian Actions on Dirac Manifolds}, Indagationes Mathematic\ae \ Volume 25, Issue 5 (2014), 901--925.
\bibitem{BursztynRadko}H.~Bursztyn, O.~Radko, \emph{Gauge equivalence of Dirac structures and symplectic groupoids}, Annales de l'Institut Fourier Volume 53, Issue 1 (2003), 309--337.
\bibitem{BursztynWeinstein_Picard}H.~Bursztyn, A.~Weinstein, \emph{Picard groups in Poisson geometry}, Moscow Math. J., 4, (2004), 39--66.
\bibitem{BursztynWeinstein}H.~Bursztyn, A.~Weinstein, \emph{Poisson geometry and Morita equivalence}, Poisson geometry, deformation quantisation and group representations, London Math. Soc. Lecture Note Series 323, (2005), 1--78.
\bibitem{BCWZ}H.~Bursztyn, M.~Crainic, A.~Weinstein, C.~Zhu, \emph{Integration of twisted Dirac brackets}, Duke Math. J. Volume 123, Number 3, (2004), 549--607.
\bibitem{H}H.~Bursztyn, \emph{A brief introduction to Dirac manifolds}, Geometric and topological methods for quantum field theory, Cambridge Univ. Press, Cambridge (2013), 4--38.
\bibitem{BursztynNosedaZhu}H.~Bursztyn, F.~Noseda, C.~Zhu, \emph{Principal actions of stacky Lie groupoids}, preprint,

\url{https://arxiv.org/abs/1510.09208}.
\bibitem{BLM}H.~Bursztyn, H.~Lima, E.~Meinrenken, \emph{Splitting theorems for Poisson and related structures}, Journal f\"ur die reine und angewandte Mathematik (Crelles Journal) (2017)

\url{https://doi.org/10.1515/crelle-2017-0014}
\bibitem{MCF}I.~Calvo, F.~Falceto, M.~Zambon, \emph{Deformation of Dirac structures along isotropic subbundles}, Rep. Math. Phys. 65 no. 2 (2010), 259--269.
\bibitem{Weins_comorph}A.~Cattaneo, B.~Dherin, A.~Weinstein, \emph{Integration of Lie algebroid comorphisms}, Portugali\ae \ Mathematica Volume 70, Issue 2, 2013, pp. 113--144.
\bibitem{Courant}T.~J.~Courant, \emph{Dirac Manifolds}, Transactions of the American Mathematical Society, Vol. 319, No. 2 (Jun., 1990), 631--661.
\bibitem{CrFer03}M.~Crainic, R.L.~Fernandes, \emph{Integrability of Lie brackets}, Annals of Mathematics 157 (2003), 575--620.
\bibitem{CrFer04}M.~Crainic, R.L.~Fernandes, \emph{Integrability of Poisson brackets}, J.~Diff.~Geom. 66 (2004), 71--137.
\bibitem{CrMar11}M.~Crainic, I.~M\u{a}rcu\cb{t}, \emph{On the existence of symplectic realizations}, J.~Symplectic~Geom. 9, no. 4 (2011), 435--444.
\bibitem{CDW}A.~Coste, P.~Dazord, A.~Weinstein, \emph{Groupo\"ides symplectiques}, Publ.~D\'ep.~Math.~Nouvelle~Ser.~A 2 (1987), 1--62.
\bibitem{PT1}P.~Frejlich, I.~M\u{a}rcu\cb{t}, \emph{The normal form theorem around Poisson transversals}, Pacific Journal of Mathematics, vol.\ 287, no.\ 2, (2017), 371--391.
\bibitem{PT1.5}P.~Frejlich, I.~M\u{a}rcu\cb{t}, \emph{Normal forms for Poisson maps and symplectic groupoids around Poisson transversals}, Lett. Math. Phys (2017)

\url{https://doi.org/10.1007/s11005-017-1007-2}
\bibitem{Marco}M.~Gualtieri, \emph{Generalized complex geometry}, Ann. of Math. (2), 174 no. 1 (2011), 75--123.
\bibitem{GinzburgLu}V.~L.~Ginzburg, J.~-H.~Lu, \emph{Poisson cohomology of Morita-equivalent Poisson manifolds}, International Mathematics Research Notices (1992), no. 10, 199--205.
\bibitem{Ginzburg}V.~L.~Ginzburg, \emph{Grothendieck Groups of Poisson Vector Bundles}, J. Symplectic Geom. Volume 1, Number 1 (2001), 121--169.
\bibitem{Karasev}M.~V.~Karas\"ev, \emph{Analogues of objects of the theory of Lie groups for nonlinear Poisson brackets}, Izv. Akad. Nauk SSSR Ser. Mat. 50 (1986), no. 3, 508--538.
\bibitem{WZW}C.~Klim\v{c}\'{i}k, T.~Str\"obl, \emph{WZW-Poisson manifolds}, Journal of Geometry and Physics 43.4 (2002), 341--344.
\bibitem{Landsman}N.~Landsman, \emph{Bicategories of operator algebras and Poisson manifolds}, Mathematical physics in mathematics and physics (Siena, 2000), vol. 30 in Fields Inst. Commun., 271--286. Amer. Math. Soc., Providence, RI, 2001.
\bibitem{Libermann}P.~Libermann, \emph{Probl\`{e}mes d'\'{e}quivalence et g\'{e}om\'{e}trie symplectique}, Ast\'{e}risque 107-108 (1983), 43--68.
\bibitem{Eckhard} E.~Meinrenken, \emph{Private communication}.
\bibitem{Eckhardnotes} E.~Meinrenken, \emph{Poisson Geometry from a Dirac perspective}, Lett. Math. Phys. (2017), 

\url{https://doi.org/10.1007/s11005-017-0977-4}
\bibitem{Park}J.-S.~Park, \emph{Topological open $p$-branes}, Symplectic Geometry and Mirror Symmetry (2000, Seoul), World Sci. Publ., River Edge, NJ (2001), 311--384.
\bibitem{Rieffel}M.~Rieffel, \emph{Morita equivalence for $C^*$-algebras and $W^*$-algebras}, J. Pure Appl. Algebra 5 (1974), 51--96.
\bibitem{Dima}D.~Roytenberg, \emph{Courant  algebroids,  derived  brackets  and  even  symplectic  supermanifolds}, PhD thesis, UC Berkeley, 1999, \url{https://arxiv.org/abs/math/9910078}.
\bibitem{SeveraWeinstein}P.~\v{S}evera, A.~Weinstein, \emph{Poisson geometry with a $3$-form background}, Progress of Theoretical Physics Supplement 144 (2002), 145--154.
\bibitem{Stefan}P.~Stefan, \emph{Accessible sets, orbits, and foliations with singularities}, Proc. London Math. Soc. (3) 29 (1974), 699--713.
\bibitem{Wade}A.~Wade, \emph{Poisson fiber bundles and coupling Dirac structures}, Ann. Global Anal. Geom. 33, no. 3 (2008), 207--217.
\bibitem{Weinstein83}A.~Weinstein, \emph{The local structure of Poisson manifolds}, J. Diff. Geom. 18 (1983), 523--557.
\bibitem{Xu0}P.~Xu, \emph{Morita equivalent symplectic groupoids}, Symplectic geometry, groupoids, and integrable systems (Berkeley, CA, 1989), 291--311. Springer, New York, 1991.
\bibitem{Xu1}P.~Xu, \emph{Morita equivalence of Poisson manifolds}, Comm. Math. Phys. Volume 142, Number 3 (1991), 493--509.
\bibitem{Xu2}P.~Xu, \emph{Momentum Maps and Morita Equivalence}, J. Differential Geom. Volume 67, Number 2 (2004), 289--333.
\end{thebibliography}
\end{document}